\documentclass[11pt]{amsart}

\usepackage{amsmath,amsthm,amsfonts,amssymb}
\usepackage[off]{auto-pst-pdf} 
\usepackage{pstricks,pst-node,pst-tree}

\usepackage{color}

\newtheorem{thm}[equation]{Theorem}
\newtheorem{cor}[equation]{Corollary}
\newtheorem{lem}[equation]{Lemma}
\newtheorem{prop}[equation]{Proposition}

\theoremstyle{definition}

\newtheorem{exam}[equation]{Example}
\numberwithin{equation}{section}

\newcommand{\depth}{\mathsf{d}}
\newcommand{\rootdeg}{\mathsf{rdeg}}
\newcommand{\comps}{\mathsf{compsize}}
\definecolor{mjo}{rgb}{0,0,.9}

\newcommand{\cc}{\mathbb{C}}
\newcommand{\z}{\mathbb{Z}}
\newcommand{\F}{\mathbb{F}}

\newcommand{\ad}{\mathsf{ad}}

\newcommand{\n}{\mathfrak{n}}
\newcommand{\h}{\mathfrak{h}}

\newcommand{\g}{\mathfrak{g}}
\newcommand{\LL}{\mathfrak{L}}

\newcommand{\T}{\mathbb{T}}

\newcommand{\der}{\mathsf{Der}}
\newcommand{\inn}{\mathsf{Inn}}
\newcommand{\aut}{{\rm Aut}}
\newcommand{\proj}{\mathsf{Proj}}

\newcommand{\rt}{\mathsf{rt}}

\definecolor{mjo}{rgb}{.4,0,.9}

\newcommand{\Hom}{\mbox{Hom}}

\begin{document}

\title[Insertion-elimination algebra: automorphisms and derivations]
{Automorphisms and derivations of the insertion-elimination algebra and related graded Lie algebras}

\author[Matthew Ondrus and Emilie Wiesner]{Matthew Ondrus and Emilie Wiesner}

\address{\noindent Mathematics Department,
Weber State University,
Ogden, UT  84408 USA,  \emph{E-mail address}: \tt{mattondrus@weber.edu}}

\address{\noindent Department of Mathematics, Ithaca College, Williams Hall Ithaca, NY 14850, USA, \ \  \emph{E-mail address}: \tt{ewiesner@ithaca.edu}}

\thanks{ \textbf{MSC Numbers (2010)}: Primary:17B65, 17B40;  Secondary: 17B70, 17B68\hfill \newline
}

\date{}

\maketitle

\begin{abstract}
This paper addresses several structural aspects of the insertion-elimination algebra $\g$, a Lie algebra that can be realized in terms of tree-inserting and tree-eliminating operations on the set of rooted trees.  In particular, we determine the finite-dimensional subalgebras of $\g$, the automorphism group of $\g$, the derivation group of $\g$, and a generating set.  Many parts of the results are stated for a more general class of Lie algebras and reproduce results for the generalized Virasoro algebras.
\end{abstract}

\section{Introduction}

This paper focuses largely on the study of the insertion-elimination algebra $\g$, a Lie algebra that can be naturally realized in terms of tree-inserting and tree-eliminating operations on the set of rooted trees. Our results include a characterization of the automorphisms, derivations, and finite-dimensional subalgebras for $\g$.

The notion of an insertion-elimination algebra was introduced by Connes and Kreimer \cite{CK} as a way of describing the combinatorics of inserting and collapsing subgraphs of Feynman graphs.  They investigated Hopf algebras related to rooted trees; the insertion-elimination algebra arises from the dual algebra of one of these Hopf algebras.  Further results on the Hopf algebra perspective have been obtained by Hoffman \cite{Hoff} and Foissy \cite{Fos}.  Sczesny \cite{S} focused on the insertion-elimination Lie algebra $\g$ under consideration in this paper, proving that $\g$ is simple as a Lie algebra and giving some fundamental results about representations for $\g$.  (We also note the papers \cite{KM04} and \cite{KM05} by Mencattini and Kreimer, investigating the ladder insertion-elimination Lie algebra.  That Lie algebra can also be characterized in terms of operations on trees; however the specific relations and resulting structure of the Lie algebra are quite different from the insertion-elimination algebra $\g$ studied here.)

Many of the results of this paper rely on properties that the insertion-elimination algebra shares with the generalized Virasoro algebras $V(M)$ introduced in \cite{PZ}, where $M$ is an additive sugroup of the underlying field $\F$.  In order to capture these similarities, we introduce (in Section \ref{subsec:genGradedLie}) the idea of a \emph{weakly triangular} decomposition for a Lie algebra $\LL= \LL^- \oplus \h \oplus \LL^+$ and the idea of a \emph{completely self-centralizing} subalgebra.

Under suitable restrictions on the underlying group $M$, it is easy to establish that the generalized Virasoro algebras $V(M)$ have a weakly triangular decomposition $V(M) = V(M)^- \oplus \h \oplus V(M)^+$.  Moreover, the weight spaces of $V(M)$ are one-dimensional, and this can be used in conjunction with Lemma \ref{lem:genNonzeroBrack-n+} to argue that the subalgebras $V(M)^{\pm}$ are completely self-centralizing.  The weight spaces for the insertion-elimination algebra $\g$ fail to have finite growth in the sense of \cite{Math92}, and the proofs of the analogous results rely on computations involving the combinatorics of rooted trees.  Proposition \ref{prop:antiAutoSymm} establishes a weakly triangular decomposition for $\g$, and Proposition \ref{prop:nonzeroBrack-n+} shows that the subalgebras $\g^{\pm}$ are completely self-centralizing.  

Using the framework outlined above, we obtain new results on derivations, automorphisms, and finite-dimensional subalgebras for the insertion-elimination algebra (as well as many of its subalgebras) and recapture results from \cite{DZ} and \cite{SZ} for generalized Virasoro algebras.   Many of our results are stated and proved in the generality of a Lie algebra $\LL$ with a regular weakly triangular decomposition $\LL^- \oplus \h \oplus \LL^+$, where the subalgebras $\LL^{\pm}$ are completely self-centralizing.  In Proposition \ref{prop:finSubsGeneral}, we show that a finite-dimensional subalgebra of $\LL$ has dimension at most $\dim \h + 2$.  In Proposition \ref{prop:autPreserveCartan} we prove  that for an $\F$-linear automorphism $\tau$ of $\LL$, $\tau (\h ) \subseteq \h$. (This fails to hold in general if $\LL^+$ is not completely self-centralizing). Proposition \ref{prop:DerDirectSum} asserts that every $\F$-linear derivation of $\LL$ is the sum of an inner derivation and a derivation of degree 0.

Using our general results for the class of Lie algebras described in Section \ref{subsec:genGradedLie}, we deduce various results specific to the insertion-elimination algebra $\g$. The finite-dimensional subalgebras of $\g$ are described in Example \ref{exam:insElFinSub}.  Theorem \ref{thm:AutInsElAlg} describes the automorphism group $\aut_\cc (\g)$ as a semidirect product, and the center of $\aut_\cc (\g)$ is characterized in Corollary \ref{cor:centerAutIE}.  In Corollary \ref{cor:DerModInIE}, we show that every derivation of $\g$ is inner, and thus the first cohomology of $\g$ with coefficients in $\g$ is trivial.   In Proposition \ref{prop:IE-NotFinGen}, we show that the Lie algebra $\g$ is not finitely generated, so that, for example, the results of \cite{Farn} cannot be used to study the derivations of $\g$.

\section{Notation and definitions}

In this section, we introduce the notions of a weakly triangular decomposition and a completely self-centralizing algebra; these definitions provide the general framework for a variety of results.  We then define the insertion-elimination algebra (along with a collection of notation for working with this algebra) and review the definition of the generalized Virasoro algebras.

\subsection{Weakly triangular decompositions and completely \\self-centralizing algebras}\label{subsec:genGradedLie}
The following definition is similar to the triangular decomposition defined in \cite{MP95} and provides a general setting for later results in the paper.  Let $\LL$ be a Lie algebra over a field $\F$.  We say that $\LL$ admits a \emph{weakly triangular} decomposition if  

\begin{enumerate}
\item $\LL = \LL^- \oplus \h \oplus \LL^+$ , for some subalgebras $\LL^{\pm}$ and $\h$, where $\h$ is abelian.

\item $\LL^+ \neq 0$, \ $[ \h, \LL^+ ] \subseteq \LL^+$, and $\LL^+$ admits a weight space decomposition relative to $\h$ (under the adjoint representation) with weights $\alpha \neq 0$ lying in a free additive semigroup $G_+ \subseteq \h^*$.  

\item There exists an anti-involution $\sigma$ on $\LL$ such that $\sigma ( \LL^+) = \LL^-$ and $\sigma |_\h = {\rm id}_\h$. 

\item There is a total ordering $>$ on $G_+$ such that for all $\alpha, \beta, \gamma \in G_+$
\begin{enumerate}
\item[(a)] $\alpha + \beta > \alpha$;  and 

\item[(b)] if $\beta \ge \gamma$, then $\alpha + \beta \ge \alpha + \gamma$.
\end{enumerate}
\end{enumerate}
Part (a) of condition (4) implies that $0 \not\in G_+$. 

For $\alpha \in \h^*$, we denote $\LL_\alpha = \{ x \in \LL \mid \mbox{$[h, x] = \alpha (h) x$ for $h \in \h$} \}$ and note that $\LL_0 = \h$.  (It may be that $\LL_\alpha = 0$ for some $\alpha \in G_+$.)   Let $G_- = \{ - \alpha \mid \alpha \in G_+ \}$ and $G = G_- \cup \{ 0 \} \cup G_+$, so that 
$$\LL^+ = \bigoplus_{\alpha \in G_+} \LL_\alpha,  \qquad \LL^- = \bigoplus_{\alpha \in G_-} \LL_\alpha, \qquad \mbox{and} \qquad \LL = \bigoplus_{\alpha \in G} \LL_\alpha.$$
It is possible that $\alpha + \beta \not\in G$ for some $\alpha \in G_-$ and $\beta \in G_+$.  However it follows from the decomposition $\LL = \LL^- \oplus \h \oplus \LL^+$ that if $[ \LL_\alpha , \LL_\beta ] \neq 0$, then $\alpha + \beta \in G$.  It is straightforward to extend the order $<$ on $G_+$ to a total order on $G$ by defining that $\alpha > \beta$ if and only if $- \alpha < - \beta$ in $G_+$, and $\alpha < 0 < \gamma$ whenever $\alpha, \beta \in G_-$ and $\gamma \in G_+$.  With this convention, then we may regard $\LL^+ = \bigoplus_{0 < \alpha \in G} \LL_\alpha$ and $\LL^- = \bigoplus_{0 > \alpha \in G} \LL_\alpha$.

If, in addition to conditions (1) through (4) holding, the weight spaces $\LL_\alpha = \{ x \in \LL \mid \mbox{$[h, x] = \alpha (h) x$ for $h \in \h$} \}$ are all finite-dimensional, then we say that $\LL$ admits a \emph{regular} weakly triangular decomposition.  

Note that (1)--(3) above are the same as (TD1)-(TD3) of \cite{MP95}, but (TD4)--stated below--has been replaced by an ordering condition on $G_+$ and $G_-$.   
\begin{enumerate}
\item[(TD4)]  There exists a basis $\{ \alpha_j \}_{j \in J}$ of $G_+$ consisting of linearly independent elements of $\h^*$.  In particular, $G_+$ consists of all nonzero finite sums of the form $\sum_{j \in J} m_j \alpha_j$ with $m_j \in \z_{\ge 0}$.  
\end{enumerate}

If the index set $J$ of (TD4) is countable, then we may use a lexicographic ordering to define a total ordering on the set $G_+$ that satisfies (4) above.  Thus we have the following proposition. 

\begin{prop}
Suppose a Lie algebra $L$ admits a triangular decomposition in the sense of \cite{MP95}.  If the set $J$ in (TD4) is countable, then $L$ admits a weakly triangular decomposition.  
\end{prop}

As we observe in Section \ref{subsec:genVir}, the generalized Virasoro algebras provide examples of Lie algebras that possess a weakly triangular decomposition but may lack a triangular decomposition.

\medskip

Let $\LL$ be a Lie algebra over a field $\F$. An element $0 \neq x \in \LL$ is {\it self-centralizing} in $\LL$ if, for $y \in \LL$, $[x,y]=0$ implies that $y \in \F x$.  We say that a Lie algebra $\LL$ such that $\dim \LL>1$ is {\it completely self-centralizing} if every nonzero element of $\LL$ is self-centralizing in $\LL$.

It is shown in \cite{BIO} that the structure of a finite-dimensional Lie algebra is greatly constrained if it contains an ad-nilpotent self-centralizing element.  The insertion-elimination algebra and the generalized Virasoro algebras contain subalgebras that are completely self-centralizing ($\LL^{\pm}$) but are infinite-dimensional and do not generally contain ad-nilpotent elements.  We show, however, that the completely self-centralizing property of $\LL^{\pm}$ strongly constrains the derivations, automorphisms, and finite-dimensional subalgebras of $\LL$.

If $\LL$ has a weakly triangular decomposition $\LL = \LL^- \oplus \h \oplus \LL^+$, the anti-involution $\sigma: \LL \to \LL$ restricts to an anti-isomorphism $\sigma : \LL^+ \to \LL^-$.  Thus we have the following.

\begin{prop}\label{prop:posIffNegative}
Let $\LL = \LL^- \oplus \h \oplus \LL^+$ be a Lie algebra with a weakly triangular decomposition.  Then $\LL^+$ is completely self-centralizing if and only if $\LL^-$ is completely self-centralizing.
\end{prop}

\subsection{The insertion-elimination algebra} \label{sec:IEnotation}

The insertion-elimination algebra is defined in terms of operations on rooted trees. We regard a rooted tree as an undirected, cycle-free graph with a distinguished vertex or {\it root}, denoted $\rt (t)$.   Let $\T$ denote the set of all (isomorphism classes of) rooted trees.  For a rooted tree $t$, $V(t)$ denotes the set of vertices of $t$, $E(t)$ denotes the set of edges of $t$, and $|t|$ denotes the cardinality of $V(t)$.  For example, if $t$ is the rooted tree \ $\psset{levelsep=0.3cm, treesep=0.3cm} 
\pstree{\Tr{$\bullet$}}{\Tr{$\bullet$} \pstree{\Tr{$\bullet$}}{\Tr{$\bullet$} \Tr{$\bullet$} \Tr{$\bullet$}}}$ (with the root displayed at the top of the picture), then $|t| = 6$ and $|E(t)| = 5$.  For $n \in \z_{>0}$, define $\T_n= \{ t \in \T \mid |t|=n\}$.

It is possible to decompose and combine rooted trees to form new rooted trees in natural, but non-unique, ways; we address a variety of ideas and notation to capture this.  If $t \in \T$, then a (rooted) {\it subtree} of $t$ is a tree $r$ such that $r$ is a connected subgraph of $t$, and $r$ is regarded as a rooted tree by declaring $\rt (r)$ to be the unique vertex of $r$ having minimal distance (in $t$) from $\rt (t)$.  In particular, if $\rt(t)$ is contained in $r$, then $\rt(r)= \rt(t)$.   We write $r \subseteq t$ to denote that $r$ is a (rooted) subtree of $t$ and regard $V(r) \subseteq V(t), E(r) \subseteq E(t)$ in the natural way.

In this paper, we will use several characteristics of a rooted tree $t \in \T$.  The \emph{depth} of $t$, denoted by $\depth (t)$, is the number of edges in the longest simple path in $t$ that begins at $\rt (t)$; and the {\it root degree} of $t$, denoted $\rootdeg(t)$, is the vertex degree of $\rt(t)$.  For $t \in \T$, the {\it components} of $t$ are the maximal subtrees $t_1, \ldots, t_k$ of $t$ such that $\rt(t_i)$ and $\rt(t)$ are connected by an edge in $t$.  (Note that $k = \rootdeg (t)$.)   We let $\comps (t) = \max \{ |r| \, \mid \mbox{$r$ is a component of $t$} \}$, i.e. the maximal size of a component of $t$.  For example, the tree $\psset{levelsep=0.3cm, treesep=0.3cm} \pstree{\Tr{$\bullet$}}{\pstree{\Tr{$\bullet$}}{\Tr{$\bullet$}} \Tr{$\bullet$} \pstree{\Tr{$\bullet$}}{\Tr{$\bullet$} \Tr{$\bullet$} \Tr{$\bullet$}}}$ contains 3 components, and $\comps (t) = 4$ in this case.

If $e \in E(t)$, then removing the edge $e$ naturally divides $t$ into two (maximal) rooted subtrees; we let $R_e(t)$ denote the subtree containing the root of $t$ and $P_e(t)$ the other subtree of $t$.  
For $s, t \in \T$ and $v \in V(s)$, we let $s \cup_v t$ denote the rooted tree obtained by joining the root of $t$ to $s$ at the vertex $v$ via a single edge and declaring that $\rt ( s \cup_v t) = \rt (s)$.  For $t_1, t_2, t_3 \in \T$, we define the following statistics: 
\begin{align*}
\alpha (t_1, t_2, t_3) &= | \{ e \in E(t_2) \mid  R_e (t_2) = t_3, \ P_e(t_2) = t_1 \} |  \\
\beta (t_1, t_2, t_3) &= | \{ v \in V(t_3) \mid t_1 = t_3 \cup_v t_2 \} |.
\end{align*}
The {\it insertion-elimination algebra} $\g$ is the Lie algebra over $\cc$ with basis $\{ d \} \cup \{ D_t^{\pm} \mid t \in \T \}$ and relations 
\begin{align*}
[D_s^+, D_t^+] &= \sum_{r \in \T} \left( \beta (r, s,t) - \beta (r,t,s)  \right) D_r^+  \\
[D_s^-, D_t^-] &= \sum_{r \in \T} \left( \alpha (t, r, s) - \alpha (s,t, r) \right) D_r^- \\
[D_s^-, D_t^+] &= \sum_{r \in \T} \alpha (s,t,r) D_r^+  + \sum_{r \in \T} \beta (s,t,r) D_r^-  \\
[D_t^-, D_t^+] &= d \\
[d, D_t^-] &= - |t| D_t^- \\
[d, D_t^+] &= |t| D_t^+,
\end{align*}
where $s, t \in \T$.  For example, $\psset{levelsep=0.3cm, treesep=0.3cm} [ D_{\pstree{\Tr{$\bullet$}}{}}^-, D_{\pstree{\Tr{$\bullet$}}{ \Tr{$\bullet$} \Tr{$\bullet$}}}^+ ] = 2 D_{\pstree{\Tr{$\bullet$}}{ \Tr{$\bullet$}}}^+$, \quad 
$\psset{levelsep=0.3cm, treesep=0.3cm} [ D_{\pstree{\Tr{$\bullet$}}{ \Tr{$\bullet$} \Tr{$\bullet$}}}^-, D_{\pstree{\Tr{$\bullet$}}{}}^+ ] = D_{\pstree{\Tr{$\bullet$}}{ \Tr{$\bullet$}}}^-$ 
and
$$\psset{levelsep=0.3cm, treesep=0.3cm}
[ D_{\pstree{\Tr{$\bullet$}}{}}^+, D_{\pstree{\Tr{$\bullet$}}{ \Tr{$\bullet$} \Tr{$\bullet$}}}^+ ] 
= D_{\pstree{\Tr{$\bullet$}}{ \Tr{$\bullet$} \Tr{$\bullet$} \Tr{$\bullet$}}}^+ \quad + \quad 2 \, D_{\pstree{\Tr{$\bullet$}}{ \Tr{$\bullet$} \pstree{\Tr{$\bullet$}}{\Tr{$\bullet$}}}}^+ \quad - \quad D_{\pstree{\Tr{$\bullet$}}{\pstree{\Tr{$\bullet$}}{ \Tr{$\bullet$} \Tr{$\bullet$}}}}^+
$$
$$\psset{levelsep=0.3cm, treesep=0.3cm}
[ D_{\pstree{\Tr{$\bullet$}}{}}^-, D_{\pstree{\Tr{$\bullet$}}{ \Tr{$\bullet$} \Tr{$\bullet$}}}^- ] = 
-3 \, D_{\pstree{\Tr{$\bullet$}}{ \Tr{$\bullet$} \Tr{$\bullet$} \Tr{$\bullet$}}}^- \quad - \quad D_{\pstree{\Tr{$\bullet$}}{ \Tr{$\bullet$} \pstree{\Tr{$\bullet$}}{\Tr{$\bullet$}}}}^- \quad + \quad D_{\pstree{\Tr{$\bullet$}}{\pstree{\Tr{$\bullet$}}{ \Tr{$\bullet$} \Tr{$\bullet$}}}}^-.
$$

For $n \in \z$, let $\g_n = \{ x \in \g \mid [d,x] = nx \}$.  Then $\g_0 = \cc d$, $\g_n = {\rm span}_\cc \{ D_t^+ \mid |t| = n \}$ for $n>0$, and $\g_n = {\rm span}_\cc \{ D_t^- \mid |t| = -n \}$ for $n < 0$.   It is clear that $\dim \g_n < \infty$ for all $n \in \z$.  When convenient, we use the notation $\h$ for the subalgebra $\g_0$, and we show in Section \ref{sec:anti-invol} that there exists an anti-involution $\sigma : \g \to \g$ satisfying $\sigma ( \g^+) = \g^-$ and $\sigma |_\h = {\rm id}_\h$.  Thus the insertion-elimination algebra admits a regular weakly triangular decomposition $\g = \g^- \oplus \h \oplus \g^+$.  (In fact, $\g$ admits a triangular decomposition in the sense of (TD1)--(TD4) of \cite{MP95}.)  We show in Proposition \ref{prop:nonzeroBrack-n+}, that $\g^+=\bigoplus_{n>0} \g_n$ and $\g^- = \bigoplus_{n<0} \g_n$ are completely self-centralizing.

\subsection{Generalized Virasoro algebras}\label{subsec:genVir}

Here we review the definition of the generalized Virasoro algebras, as introduced in \cite{PZ}.

Let $\F$ be a field and $M$ an additive subgroup of $\F$.  Then the {\it generalized Virasoro algebra} $V(M)$ has a basis $\{z \} \cup \{e_\alpha \mid \alpha \in M\}$ and relations
\begin{align*}
[e_\alpha, e_\beta] &= (\beta-\alpha)e_{\alpha+\beta} + \beta^3 \delta_{\alpha, -\beta} z  \\
[z, e_\alpha]&=0
\end{align*}
for all $\alpha, \beta \in M$.  It is clear from the definition of $V(M)$ that it is graded by the group $M$.  In Section 2 of \cite{PZ}, the notion of an additive total ordering on $M$ (i.e. a total ordering on $M$ for which the sum of two positive elements of $M$ is positive) is used to define a decomposition $V(M) = V(M)_+ \oplus V(M)_0 \oplus V(M)_-$, where $V(M)_+ = {\rm span}_\F \{ e_\alpha \mid \alpha > 0 \}$, $V(M)_- = {\rm span}_\F \{ e_\alpha \mid \alpha < 0 \}$, and $V(M)_0 = \F e_0 \oplus \F z$ are subalgebras.  

The weight spaces of $V(M)$ are one-dimensional, and $V(M)$ possesses an anti-involution $\sigma : V(M) \to V(M)$ where $\sigma (e_\alpha ) = e_{-\alpha}$ and $\sigma (z) = z$.  If $M$ has an additive total ordering, we may regard the set $G_+ = \{ \alpha \in M \mid \alpha > 0 \}$ as a subset of $V(M)_0^*$ by identifying $\alpha \in M$ with the map that sends $e_0$ to $\alpha$ and sends $z$ to $0$.  Then it is clear that $V(M)$ admits a regular weakly triangular decomposition.  However, $V(M)$ may not have a triangular decomposition in the sense of \cite{MP95}.  For example, if $M=\mathbb R$, then (TD4) of \cite{MP95} does not hold.

It is straightforward to use the total ordering on $M$ and the fact that the weight spaces of $V(M)$ are one-dimensional to show that $V(M)_{\pm}$ is completely self-centralizing.

\section{An anti-involution for the insertion-elimination algebra}\label{sec:anti-invol}

In this section, we present an anti-involution $\sigma$ of the insertion-elimination algebra $\g$ such that $\sigma(d)=d$ and $\sigma(\g^\pm) = \g^\mp$. Equipped with this map, $\g$ possesses a weakly triangular decomposition. Additionally, the anti-involution is useful in determining all automorphisms for $\g$ in Section \ref{sec:automorph}.

\begin{lem}\label{lem:antiAutSuffic} 
Let $\sigma : \g \to \g$ be a linear function such that $\sigma(d)=d$.  Assume that for each $t \in \T$, there is $0 \neq \mu_t \in \cc$, such that $\sigma (D_t^+) = \mu_t D_t^-$, and $\sigma (D_t^-) = \mu_t^{-1} D_t^+$.   If for all $r, s, t \in \T$,  $\beta (t, s, r) \mu_t = \alpha (s, t, r ) \mu_r \mu_s,$ then $\sigma$ is an anti-involution.
\end{lem}
\begin{proof}
Suppose that the set $\{\mu_t \mid t \in \T \}$ satisfies $\beta (t, s, r) \mu_t = \alpha (s, t, r ) \mu_r \mu_s$. It is enough to show that $\sigma$ satisfies $\sigma ([a,b]) = [ \sigma (b), \sigma (a) ]$ for all $a, b$ belonging to the standard basis of $\g$.  To do this, we consider the following cases.

For $r, s \in \T$, we have 
\begin{align*}
\sigma ([ D_r^+, D_s^+]) &= \sigma \left( \sum_{t \in \T} ( \beta (t, r, s) - \beta (t, s, r)) D_t^+ \right)\\
& = \sum_{t \in \T} ( \beta (t, r, s) - \beta (t, s, r)) \mu_t D_t^-;\\
[ \sigma (D_s^+), \sigma (D_r^+)] &= \mu_s \mu_r [ D_s^-, D_r^- ] = \mu_s \mu_r \left( \sum_{t \in \T} ( \alpha (r, t, s) - \alpha (s,t,r) ) D_t^- \right).
\end{align*}
To show that $\sigma ([ D_r^+, D_s^+]) = [ \sigma (D_s^+), \sigma (D_r^+)]$, it is sufficient to show that for a given $t \in \T$, 
$$( \beta (t, r, s) - \beta (t, s, r)) \mu_t  = \mu_s \mu_r ( \alpha (r, t, s) - \alpha (s,t,r) ).$$
By assumption, we have that $\beta (t, s, r) \mu_t = \alpha (s, t, r ) \mu_r \mu_s$, and $\beta (t, r, s) \mu_t = \alpha (r, t, s ) \mu_r \mu_s$; the assertion that $\sigma ([ D_r^+, D_s^+]) = [ \sigma (D_s^+), \sigma (D_r^+)]$ follows.

We next show that $\sigma ( [D_r^-, D_s^+]) = [ \sigma (D_s^+), \sigma (D_r^-) ]$.  We have 
\begin{align*}
\sigma ([ D_r^-, D_s^+]) &=\sigma \left( \sum_{t \in \T} \alpha (r, s, t) D_t^+ + \sum_{t \in \T} \beta (r, s, t) D_t^- \right)\\
&= \sum_{t \in \T} \alpha (r, s, t) \mu_t D_t^- + \sum_{t \in \T} \beta (r, s, t) \mu_t^{-1} D_t^+;\\
[ \sigma (D_s^+), \sigma (D_r^-)] &= \mu_s \mu_r^{-1} [ D_s^-, D_r^+ ] \\
&= \mu_s \mu_r^{-1} \left( \sum_{t \in \T} \alpha (s,r,t) D_t^+ + \sum_{t \in \T} \beta (s,r,t) D_t^- \right).
\end{align*}
To show that $\sigma ([ D_r^-, D_s^+]) = [ \sigma (D_s^+), \sigma (D_r^-)]$, it is enough to show that $\beta (r,s,t) \mu_t^{-1} = \alpha (s,r,t) \mu_s \mu_r^{-1}$ and $\alpha (r, s, t) \mu_t = \beta (s,r,t) \mu_s \mu_r^{-1}$.  These are equivalent to $\beta (r,s,t) \mu_r = \alpha (s,r,t) \mu_s \mu_t$ and $\alpha (r, s, t) \mu_t \mu_r = \beta (s,r,t) \mu_s$, which hold by assumption. 

The remaining cases are similar. 
\end{proof}

The anti-involution for $\g$ is defined in terms of symmetries (i.e. graph automorphisms) of rooted trees.  In \cite{Hoff}, Hoffman used symmetries of rooted trees to define an inner product on the graded vector space spanned by rooted trees.  Under this inner product, Hoffman's growth and pruning operators, which act like $D^\pm_{\pstree{\Tr{$\bullet$}}{}}$, are adjoint operators.  

We define a {\it graph automorphism} of a rooted tree $t$ as a bijection $\tau : V(t) \to V(t)$ that fixes the root of $t$ and has the property that, for $v_1, v_2 \in V(t)$, $(v_1, v_2)$ is an edge of $t$ if and only if $( \tau (v_1), \tau (v_2) )$ is an edge of $t$.  For example, the rooted tree 
$$
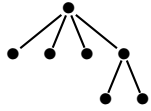
\psset{levelsep=0.6cm, treesep=0.3cm}
\pstree{\Tr{$\bullet$}}{\Tr{$\bullet$} \Tr{$\bullet$} \Tr{$\bullet$} \pstree{\Tr{$\bullet$}}{\Tr{$\bullet$} \Tr{$\bullet$}}}
$$
has $12 = (3!)(2!)$ automorphisms.  

Define 
$$\Gamma_t= \{ \tau : V(t) \to V(t) \mid \tau \ \mbox{a graph automorphism of $t$} \},$$ 
 the {symmetry group} of $t$; and
\begin{equation}\label{eqn:treeAutSize}
\xi_t = | \Gamma_t|,
\end{equation}
the number of graph automorphisms of $t$.

\begin{lem}\label{lem:graphAutsWork}
For $r, s, t \in \T$, 
$$\beta (t, s, r) \xi_t = \alpha (s, t, r ) \xi_r \xi_s.$$
\end{lem}
\begin{proof}
First suppose that $r, s, t \in \T$ are such that $\alpha (s,t,r) \neq 0$ and $\beta (t,s,r) \neq 0$.  Then  the claim is equivalent to $$\frac{\xi_r \xi_s}{\beta (t, s, r)} = \frac{\xi_t}{\alpha (s, t, r )}.$$ 
If $\alpha (s,t,r) \neq 0$, we may choose a vertex $v_0 \in V(r)$ with the property that $r \cup_{v_0} s = t$.  Note that the $\Gamma_r$-orbit of the vertex $v_0 \in V(r)$ consists of all vertices $v \in V(r)$ with the property that $r \cup_{v_0} s = t$; therefore, the size of the orbit is given by $\beta (t,s,r)$.  Then the orbit-stabilizer theorem implies that $\frac{\xi_r}{\beta (t,s,r)} = \frac{|\Gamma_r|}{\beta (t,s,r)}$ is the size of the stabilizer of $v_0$ in $\Gamma_r$, i.e. the number of automorphisms of $r$ that fix $v_0$.  

Now treating $v_0$ as an element of $V(t)$, we have that the $\Gamma_t$-orbit of $v_0$ consists of all edges $e$ with the property that $R_e(t) = r$ and $P_e(t) = s$; thus, $\alpha (s,t,r)$ counts the size of this orbit.  Then the orbit-stabilizer theorem implies that $\frac{\xi_t}{\alpha (s,t,r)} = \frac{|\Gamma_t|}{\alpha (s,t,r)}$ is the size of the stabilizer of $v_0$ in $\Gamma_t$, i.e. the number of automorphisms of $t$ that fix $v_0$.  

Note the natural bijection 
$$\left\{ \begin{array}{ll} \text{automorphisms} \\ \text{of $r$ that fix $v_0$} \end{array} \right\} \times \left\{ \begin{array}{ll} \text{automorphisms} \\ \text{of $s$} \end{array} \right\} \longleftrightarrow \left\{ \begin{array}{ll} \text{automorphisms} \\ \text{of $t$ that fix $v_0$} \end{array} \right\}.$$
Taking the cardinalities of each of the sets involved, we have 
$$\frac{\xi_r}{\beta (t,s,r)} \cdot \xi_s = \frac{\xi_t}{\alpha (s,t,r)}.$$

Now, it suffices to show that $\beta (t, s, r) \neq 0$ if and only if $\alpha (s, t, r) \neq 0$.  
Suppose that $\alpha (s,t,r) \neq 0$.  Then there exists $e \in E(t)$ such that $R_e(t) = r$ and $P_e(t) = s$, and we may therefore view $r$ as a subset of $t$.  If we let $v$ denote the vertex of $r$ on which $e$ is incident, then $t = r \cup_v s$, and it follows that $\beta (t,s,r) \neq 0$.  A similar argument shows that if $\beta (t,s,r) \neq 0$, then $\alpha (s,t,r) \neq 0$.
\end{proof}

From Lemma \ref{lem:antiAutSuffic} and Lemma \ref{lem:graphAutsWork}, we now have the following.

\begin{prop}\label{prop:antiAutoSymm}
Let $\sigma : \g \to \g$ be the linear function given by $\sigma (d) = d$ and 
$$\sigma (D_t^+) = \xi_t D_t^- \qquad \mbox{and} \qquad \sigma (D_t^-) = \xi_t^{-1} D_t^+$$
for $t \in \T$, where $\xi_t$ is as in (\ref{eqn:treeAutSize}).   Then, $\sigma$ is an anti-involution of $\g$.
\end{prop}

This result implies that $\g$ has a regular weakly triangular decomposition (as well as a triangular decomposition in the sense of \cite[p.~95]{MP95}).  Also, we note that the existence of the anti-involution of Proposition \ref{prop:antiAutoSymm} implies the existence of the Shapovalov determinant, and Theorem 3.2 of \cite{S} is equivalent to the claim that the Shapovalov determinant (considered over all positive root spaces) has at most countably many zeros.

\section{Completely self-centralizing subalgebras of the insertion-elimination algebra}\label{sec:centralizeIEalg}

In this section, we show that the subalgebras $\g^\pm$ of the insertion-elimination algebra $\g$ are completely self-centralizing. (See Section \ref{subsec:genGradedLie} for definitions.)  The following lemma is used to reduce the result to a computation involving elements that are homogeneous with respect to the grading on $\g$.  The result can be stated more generally in terms of a graded Lie algebra.  However, for clarity we state the lemma in the context in which it will be used.

\begin{lem}\label{lem:genNonzeroBrack-n+}
Let $\LL = \LL^- \oplus \h \oplus \LL^+$ be a Lie algebra over $\F$ admitting a weakly triangular decomposition as defined in Section \ref{subsec:genGradedLie}.  Then the subalgebra $\LL^+$ is completely self-centralizing if and only if $[x_\alpha,y_\beta] \neq 0$ for every linearly independent pair of vectors $x_\alpha \in \LL_\alpha$, $y_\beta \in \LL_\beta$ ($\alpha, \beta \in G_+$).  

In particular, suppose that $[x_\alpha,y_\beta] \neq 0$ for every linearly independent pair of vectors $x_\alpha \in \LL_\alpha$, $y_\beta \in \LL_\beta$.  For linearly independent elements $x = \sum_{\alpha \in G_+} x_\alpha , \   y = \sum_{\alpha \in G_+} y_\alpha \in \LL^+$ with $x_\alpha, y_\alpha \in \LL^+_\alpha$, let $\nu = \max \{ \alpha \mid x_\alpha \neq 0  \}$ and $\mu = \max \{ \alpha \mid y_\alpha \neq 0 \}$, and assume $\mu \ge \nu$.  Then 
$$\proj_{\LL_{\nu+\kappa}} [x,y] \neq 0,$$
where $\kappa \in G_+$ is such that there is $c \in \F$ with $\sum_{\beta > \kappa} y_\alpha = c \sum_{\alpha > \kappa} x_\beta$ but $y_{\kappa} \neq c x_{\kappa}$. 
\end{lem}

\begin{proof}
Clearly, if $\LL^+$ is completely self-centralizing, then  $[x_\alpha, y_\beta] \neq 0$ whenever $x_\alpha \in \LL_{\alpha}$ and $y_\beta \in \LL_\beta$ ($\alpha, \beta \in G_+$) are linearly independent.  To show the other direction, it is sufficient to prove the specific computational assertion.

Therefore, suppose $[x_\alpha,y_\beta] \neq 0$ for every linearly independent pair of vectors $x_\alpha \in \LL_\alpha$, $y_\beta \in \LL_\beta$.  Let $x,y \in \LL^+$ be arbitrary linearly independent elements, written in the form stated in the claim, and also define $\mu \ge \nu$ as in the statement of the lemma.   If $\mu > \nu$ or if $\mu=\nu$ and $\kappa = \nu$, the result is obvious.   Thus we assume below that $\mu = \nu$ and $\kappa < \nu$.

Let $\hat x = \sum_{\alpha \le \kappa} x_\alpha$, \, $\tilde x = \sum_{\alpha > \kappa} x_\alpha$, \, $\hat y = \sum_{\alpha \le \kappa} y_\alpha$, and $\tilde y = \sum_{\alpha > \kappa} y_\alpha$.  Note that $\tilde y = c \tilde x$ for some $c \in \F$, so $[ \tilde x, \tilde y]= 0$. Also, since $\kappa<\nu$, we have that $\proj_{\LL_{\kappa+ \nu}} [\hat x, \hat y]=0$. Therefore, 
\begin{align*}
\proj_{\LL_{\kappa+ \nu}}([x,y]) &= \proj_{\LL_{\kappa+ \nu}}([ \hat x, \tilde y] + [ \tilde x, \hat y])=
[x_\kappa, y_\nu] + [x_\nu, y_\kappa] \\&= [x_\kappa, c x_\nu] + [x_\nu, y_\kappa] = [x_\nu, y_\kappa- cx_\kappa].
\end{align*}
 From our choice of $\kappa$, we know $0 \neq y_\kappa - cx_\kappa \in \LL_\kappa$ and $y_\kappa-c x_\kappa, x_\nu$ are linearly independent. Therefore, $[x_\nu, y_\kappa- cx_\kappa] \neq 0$ by assumption.
\end{proof}

\begin{lem} \label{lem:nonzeroBrackWtSp}
Let $\g$ denote the insertion-elimination algebra, and let $x \in \g_m$ and $y \in \g_n$, where $0 < m, n \in \z$.  If $x, y$ are linearly independent, then $[x,y] \neq 0$. 
\end{lem}

We use the following notation in the proof of Lemma \ref{lem:nonzeroBrackWtSp}, building on the notation from Section \ref{sec:IEnotation}.  For $r, t \in \T$, $v \in V(r)$, and $0 < m \in \z$, define $r \cup_v^m t$ to be the rooted tree formed by attaching $m$ copies of $t$ to $r$, with each copy connected by a single edge from $v$ to the root of that copy of $t$.   For example, 
$$
r = \psset{levelsep=0.4cm, treesep=0.4cm}
\pstree{\Tr{$\bullet$}}{\Tr{$v \bullet$\,} \Tr{$\bullet$} \pstree{\Tr{$\bullet$}}{\Tr{$\bullet$}}},
\quad 
t = \pstree{\Tr{$\bullet$}}{\pstree{\Tr{$\bullet$}}{\Tr{$\bullet$}} \pstree{\Tr{$\bullet$}}{\Tr{$\bullet$}}}
\psset{levelsep=0.6cm, treesep=0.3cm}
\quad 
\Longrightarrow
\quad
r \cup_v^3 t =  
\pstree{\Tr{$\bullet$}}{\pstree{\Tr{\, $\bullet$ \,}}
{\pstree{\Tr{$\bullet$}}{\pstree{\Tr{$\bullet$}}{\Tr{$\bullet$}} \pstree{\Tr{$\bullet$}}{\Tr{$\bullet$}}} 
\pstree{\Tr{$\bullet$}}{\pstree{\Tr{$\bullet$}}{\Tr{$\bullet$}} \pstree{\Tr{$\bullet$}}{\Tr{$\bullet$}}} 
\pstree{\Tr{$\bullet$}}{\pstree{\Tr{$\bullet$}}{\Tr{$\bullet$}} \pstree{\Tr{$\bullet$}}{\Tr{$\bullet$}}}}  
\Tr{$\bullet$} 
\pstree{\Tr{$\bullet$}}{\Tr{$\bullet$}}}.
$$
Define $r \cup_v^0 t = r$.
Additionally, for $s, t \in \T$ and $v \in V(s)$, define 
$$M( s, t, v) = \max \{ m \in \z \mid s = r \cup_v^m t \ \mbox{for some $r \subseteq s$ and $v \in V(r ) \subseteq V(s)$} \}.$$
In particular, if there does not exist $e \in E(s)$ incident on $v$ with $P_e(s) = t$, then $M(s,t,v) = 0$.  

\begin{proof}
Since the Lie bracket is antisymmetric, it is no loss to assume that $m \geq n$.  Let $\{D_{s_1}^+, \ldots, D_{s_k}^+ \}$ be a basis for $\g_m$ and $\{ D_{t_1}^+, \ldots, D_{t_{\ell}}^+ \}$ be a basis for $\g_n$; if $m=n$, we assume the two bases coincide.  Writing
$$x = \sum_{1 \leq i \leq k} b_i D_{s_i}^+, \quad y = \sum_{1 \leq j \leq \ell} c_i D_{t_i}^+, \quad (b_i, c_j \in \cc),$$ we have
$$
[x,y]= \sum_{i,j} b_ic_j [D_{s_i}^+, D_{t_i}^+].
$$

First consider the case $m>n$.   Let 
$$M = \max \{ M(s_i, t_j, v) \mid 1 \le i \le k, \, 1 \le j \le \ell, b_ic_j \neq 0, v \in V(s_i) \}.$$
After possibly renumbering, it is no loss to assume that $b_1, c_1 \neq 0$ and $M = M(s_1, t_1, v_1)$ for some $v_1 \in V(s_1)$.

Using the definition of $\beta(r, s, t)$, we may rewrite the product $[x,y]$ as
$$
[x,y]= \sum_{i,j} b_ic_j [D_{s_i}^+, D_{t_j}^+] = \sum_{i, j} b_ic_j \left(\sum_{v \in V(t_j)} D^+_{t_j \cup_v s_i} - \sum_{v \in V(s_i)} D^+_{s_i \cup_v t_j} \right).
$$
To show that $[x,y] \neq 0$, we argue that $D^+_{s_1 \cup_{v_1} t_1}$ has a nonzero coefficient in $[x,y]$.  It is enough to show that $s_1 \cup_{v_1} t_1 \neq u$ for any other $u$ such that $D_u^+$ appears with a nonzero coefficient as a summand above.

First suppose that $s_1 \cup_{v_1} t_1 = s_i \cup_w t_j$ for some $i, j$ and $w \in V(s_i)$.  Then there exists $w' \in V(s_i \cup_w t_j)$ such that $M(s_i \cup_w t_j, t_1, w')=M+1$: that is, $M+1$ copies of $t_1$ are attached at $w'$.  If $w' \in V(t_j) \subseteq V(s_i \cup_w t_j)$, this forces $t_1$ to be a proper subtree of $t_j$, which is impossible as $|t_1| = |t_j|$.  Therefore, it must be that $w' \in V(s_i) \subseteq V(s_i \cup_w t_j)$.  However, this implies that $M(s_i, t_1, w')=M+1$, a violation of the maximality of $M$ unless $w'=w$ and $t_j=t_1$.  We now have $s_1 \cup_{v_1} t_1 = s_i \cup_w t_1$ and again the maximality of $M$ forces $s_i = s_1$.

Suppose instead that $s_1 \cup_{v_1} t_1 = t_j \cup_w s_i$ for some $i, j$ and $w \in V(t_j)$.  Let $w_1 \in V(t_1) \subseteq V( s_1 \cup_{v_1} t_1)$ be the root of $t_1$, connected to $s_1$ via the edge $e_1=(v_1, w_1) \in E( s_1 \cup_{v_1} t_1)$; let $v_1', w_1' \in V(t_j \cup_w s_i)$ and $e_1' \in E(t_j \cup_w s_i)$ be the corresponding vertices and edge under the isomorphism between $s_1 \cup_{v_1} t_1$ and $t_j \cup_w s_i$.  Note that $|P_{e_1'}(t_j \cup_w s_i)| =|P_{e_1}(s_1 \cup_{v_1} t_1)|= |t_1|$.  Since $|t_j| = |t_1| < |s_i|$, this forces $w_1' \in V(s_i) \subseteq V(t_j \cup_w s_i)$.
Now, if $v_1' \in V(s_i) \subseteq V(t_j \cup_w s_i)$, then $M(s_i, t_1, v_1')=M+1$; this violates the maximality of $M$.  Therefore, $v_1' \in V(t_j)$.  Since the only edge joining $t_j$ and $s_i$ is $(w, rt(s_i))$, it must be that $w_1'=w$ and $w_1'=rt(s_i)$.  Since $|P_{e_1'}(t_j \cup_w s_i)|=|P_{e_1} (s_1 \cup_{v_1} t_1)|=|t_1|$, this forces $|s_i| = |t_1|$. However, this contradicts the assumption $m>n$, so it cannot be that $s_1 \cup_{v_1} t_1 = t_j \cup_w s_i$.

Now consider the case $m=n$.  Using the common basis $\{D_{s_1}^+, \ldots, D_{s_k}^+ \}$ for $\g_m$, we have 
$$
0 = [x,y] = \sum_{i,j}  b_i c_j [ D_{s_i}^+, D_{s_j}^+ ] = \sum_{i<j} (b_i c_j - b_j c_i) [ D_{s_i}^+, D_{s_j}^+ ].
$$
Because $x$ and $y$ are linearly independent, it is no loss to assume (after possibly reordering $\{s_1, \ldots, s_k\}$) that $b_1c_2-b_2c_1 \neq 0$.   For a fixed $v_0 \in V(s_1)$, we argue that the coefficient of $D^+_{s_1 \cup_{v_0} s_2}$ is nonzero in $[x,y]$.    Suppose that $s_1 \cup_{v_0} s_2  = s_i \cup_{v'} s_j$ for some $(i,j) \neq (1,2)$ and some $v' \in V(s_i)$.  Let $e \in E(s_1 \cup_{v_0} s_2)$ be the edge connecting $s_2$ to $s_1$ (at $v_0$), $e' \in E(s_i \cup_{v'} s_j)$ be the edge connecting $s_j$ to $s_i$ (at $v'$), and $e'' \in E(s_i \cup_{v'} s_j)$ the image of $e$ under the isomorphism $s_1 \cup_{v_0} s_2  = s_i \cup_{v'} s_j$.  Note that $|P_{e''} (s_i \cup_{v'} s_j)| = |P_{e} (s_1 \cup_{v_0} s_2)| =|s_2|$. However, since $|s_2| = |s_i| = |s_j|$, it is clear that $e' \in E(s_i \cup_{v'} s_j)$ is the only edge so that $|P_{e'} (s_i \cup_{v'} s_j)| = |s_2|$. Therefore, $e'=e''$ and $v_0$ must map to $v'$ under the isomorphism $s_1 \cup_{v_0} s_2  = s_i \cup_{v'} s_j$.  This forces $s_1  = s_i$ and $s_2  = s_j$, and thus the coefficient of $D_{s_1 \cup_{v_0} s_2}^+$ is nonzero in $[x,y]$.
\end{proof}

\begin{prop}\label{prop:nonzeroBrack-n+}
The subalgebras $\g^\pm$ of the insertion-elimination algebra $\g$ are completely self-centralizing.
\end{prop}
\begin{proof}
The result for $\g^+$ follows from Lemmas \ref{lem:genNonzeroBrack-n+} and \ref{lem:nonzeroBrackWtSp}. Proposition \ref{prop:posIffNegative} then implies $\g^-$ is completely self-centralizing.
\end{proof}

\section{Finite-dimensional subalgebras}

In this section, we investigate the finite-dimensional subalgebras of a Lie algebra $\LL$ with weakly triangular decomposition and then specialize to the insertion-elimination algebra $\g$.  These results play a role in the study of automorphisms in Section \ref{sec:automorph}.

Proposition \ref{prop:finSubsGeneral} does not depend on all the conditions of a weakly triangular decomposition (see Section \ref{subsec:genGradedLie}).  However, we use this terminology in the statements in order to maintain consistency in the assumptions used for our results.  Note that from Proposition \ref{prop:posIffNegative}, we need only assume that $\LL^+$ is completely self-centralizing. 

\begin{prop}\label{prop:finSubsGeneral}
Let $\LL = \bigoplus_{\alpha \in G} \LL_\alpha$ be a Lie algebra admitting a regular weakly triangular decomposition over $G$.  Assume that $\LL^+=\bigoplus_{\alpha \in G_+} \LL_{\alpha}$ is completely self-centralizing and $\dim \h =k$.   If $\mathfrak s$ is a finite-dimensional Lie subalgebra of $\LL$, then $\dim \mathfrak s \le k+2$.  Moreover, if $\dim \mathfrak s = k+2$, then $\mathfrak s= \F y \oplus \h \oplus \F x$ for some $ \alpha, \beta \in G_+$, $x \in \LL_\alpha$, $y \in \LL_{-\beta}$.
\end{prop}
\begin{proof}
First consider a set $\{ a_1, a_2, a_3, \ldots, a_l \}$ of linearly independent vectors in $\LL$.  We will show that $\mbox{span}_\F \{ a_1, a_2, a_3, \ldots, a_l \}$ is not closed under commutators if either $l>k+2$ or if $l=k+2$ and $\h \not\subseteq \mbox{span}_\F \{ a_1, a_2, a_3, \ldots, a_l \}$.  Thus, in these cases there cannot be a subalgebra $\mathfrak s$ of dimension $l$.

\smallskip

For each $1 \leq i \leq l$, write $a_i = y_i + h_i + x_i$, with $y_i \in \LL^-$, $h_i \in \h$, and $x_i \in \LL^+$; and suppose $\dim (\mbox{span}_\F \{ x_1, \ldots, x_l \} ) \geq 2$ (that is, there is at least one linearly independent pair among the $x_i$). For each $i$, write $x_i = \sum_{\beta \in G_+} X_{i,\beta}$ with $X_{i,\beta} \in \LL_\beta$, and let $\nu_i \in G_+$ be maximal such that $X_{i, \nu_i} \neq 0$.  

For a pair of linearly independent $x_i, x_j$, let $\kappa_{i,j}$ be the unique element of $G_+$ such that 
$$
\mbox{ $\sum_{\beta > \kappa_{i,j}} X_{j, \beta} = c \sum_{\beta > \kappa_{i,j}} X_{i, \beta}$ and $X_{j, \kappa_{i,j}} \neq c X_{i, \kappa_{i,j}}$ for some $c \in \F^*$.}
$$ 
Set $M=\max \{ \nu_i+ \kappa_{i,j} \mid x_i, x_j \ \mbox{linearly independent} \}$ and choose $i_0, j_0$ such that $\nu_{i_0}+\kappa_{i_0, j_0}=M$.  Then Lemma \ref{lem:genNonzeroBrack-n+} implies that $\proj_{\LL_M} [x_{i_0}, x_{j_0}] \neq 0$.  Therefore it's enough to show that $X_{i, M} = 0$ for all $i \in \{ 1, \ldots, l \}$, since this implies that $[a_{i_0}, a_{j_0}] \not\in {\rm span}_\F \{ a_1, a_2, a_3, \ldots, a_l \}$.  If $X_{i', M} \neq 0$ for some $i'$, then since $M > \nu_{i_0}$, it follows that $x_{i_0}$ and $x_{i'}$ are independent and $\kappa_{i_0, i'} \ge M$.  But this implies that $\nu_{i_0} + \kappa_{i_0, i'} \ge \nu_{i_0} + \nu_{i_0} + \kappa_{i_0, j_0} = \nu_{i_0} + M > M$ (by condition (4) of the definition of a weakly triangular decomposition), contradicting the maximality of $M$.

We now determine conditions that guarantee that there is at least one pair $x_i, x_j$ of linearly independent vectors (or, by similar arguments, a pair $y_i, y_j$ of linearly independent vectors) as defined above.  If $l>k+2$, this is clearly the case.  If $l=k+2$, such a pair exists unless $\h \subseteq \mbox{span}_\F \{ a_1, a_2, a_3, \ldots, a_l \}$ and all $x_i$ are scalar multiples of each other and similarly all $y_i$ are scalar multiples of each other. 

From the above argument, we may conclude the following regarding a Lie subalgebra $\mathfrak s \subseteq \LL$.  If $\dim \mathfrak s > k+2$, then $\dim \mathfrak s = \infty$.  If $\dim \mathfrak s = k+2$, then $\mathfrak h \subseteq \mathfrak s$.  To complete the proof, it remains to show, in the case that $\dim \mathfrak s = k+2$, that $\mathfrak s= \F y \oplus \h \oplus \F x$, where $ \alpha, \beta \in G_+$, $x \in \LL_\alpha$, $y \in \LL_{-\beta}$.  

If $\dim \mathfrak s = k+2$, let $\{ a_1, \ldots , a_l \}$ be a spanning set for $\mathfrak s$, and write $a_i = y_i + h_i + x_i$ as above.  Since $\h \subseteq \mathfrak s = {\rm span} \{ a_1, \ldots, a_l \}$, we may assume (after possibly relabeling and taking linear combinations) that $a_1= x_1 \neq 0$ and $a_2=y_2 \neq 0$. Note that if there is no $\alpha>0$ such that $x_1 \in \LL_\alpha$, then there is some $h \in \h \subseteq \mathfrak s$ such that $[h, x_1] \not\in \cc x_1$ and thus $[h, x_1] \not\in \mathfrak s$.  Therefore $x_1 \in \LL_\alpha$ for some $\alpha \in G_+$.  We can similarly argue that $y_2 \in \LL_{-\gamma}$ for some $\gamma \in G_+$. 
\end{proof}

\begin{exam}\label{exam:insElFinSub}
Let $\mathfrak s$ be a finite-dimensional subalgebra of the insertion-elimination Lie algebra $\g = \g^- \oplus \h \oplus \g^+$.  Proposition \ref{prop:finSubsGeneral} implies that either 
\begin{itemize}
\item $\dim \mathfrak s \le 2$; or 

\item $\dim \mathfrak s = 3$ and there exist elements $x \in \g_m$, $y \in \g_{-n}$, where $m, n > 0$, such that $\mathfrak s = \cc x \oplus \cc d \oplus \cc y$.
\end{itemize}

Note that $\g$ contains both abelian subalgebras of dimension $2$ as well as subalgebras isomorphic to the $2$-dimensional non-abelian subalgebra.  If $\dim \mathfrak s = 3$, with $x$ and $y$ as above, then there are two possibilities.  If $[x,y] = 0$ (e.g. if $\psset{levelsep=0.3cm, treesep=0.3cm} x = D_{\pstree{\Tr{$\bullet$}}{ \Tr{$\bullet$}}}^+$
and $\psset{levelsep=0.3cm, treesep=0.3cm} y = D_{\pstree{\Tr{$\bullet$}}{ \Tr{$\bullet$} \Tr{$\bullet$}}}^-$), then $[ \mathfrak s, \mathfrak s] = \cc x \oplus \cc y$ is 2-dimensional.  If $[x, y] \neq 0$, then it must be that $m = n$ and $\mathfrak s \cong \mathfrak{sl}_2( \cc )$.
\end{exam}

\begin{exam}\label{exam:genVirSubAlg}
For a field $\F$ and an additive subgroup $M \subseteq \F$, let $V(M)$ denote the generalized Virasoro algebra reviewed in Section \ref{subsec:genVir}.  If there is an additive total order on the group $M$, then $V(M)$ satisfies the assumptions of Section \ref{subsec:genGradedLie}, and it is straightforward to use Lemma \ref{lem:genNonzeroBrack-n+} to show that the subalgebras $V(M)^+ = \bigoplus_{\alpha > 0} \F e_\alpha$ and $V(M)^- = \bigoplus_{\alpha < 0} \F e_\alpha$ are completely self-centralizing.   Thus Proposition \ref{prop:finSubsGeneral} implies that a finite-dimensional subalgebra of $V(M)$ can have dimension at most 4, and a 4-dimensional subalgebra must have the form $\F e_{-\alpha} \oplus ( \F e_0 \oplus \F z ) \oplus \F e_\alpha$.   This reproduces results given in \cite{SZ} (see Lemma 3.1 of \cite{SZ}).
\qed
\end{exam}

\section{Automorphism groups} \label{sec:automorph}

In this section we prove that for a Lie algebra with a regular weakly triangular decomposition, automorphisms preserve the Cartan subalgebra.  We then specialize this result to the insertion-elimination algebra.

\begin{prop}\label{prop:autPreserveCartan}
Let $\LL = \bigoplus_{\alpha \in G} \LL_\alpha$ be a Lie algebra admitting a regular weakly triangular decomposition, and assume that $\LL^+=\bigoplus_{\alpha \in G_+} \LL_{\alpha}$ is completely self-centralizing.  If $\tau$ is a Lie algebra automorphism of $\LL$, then $\tau (\h) \subseteq \h$ and $\tau ( \LL_\alpha ) \subseteq \LL_{\alpha \circ \tau^{-1}}$.
\end{prop}
\begin{proof}
Let $\mathsf T$ denote the intersection of all subalgebras of $\LL$ of dimension $\dim \h + 2$.  Proposition \ref{prop:finSubsGeneral} implies that $\h \subseteq \mathsf T$.

Since $\LL^+$ is completely self-centralizing, there must exist $\alpha, \beta \in G_+$ with $\alpha \neq \beta$ such that $\LL_\alpha, \LL_\beta \neq 0$.  Let $0 \neq X_\alpha \in \LL_\alpha$ and $0 \neq X_\beta \in \LL_\beta$.  If $\sigma : \LL \to \LL$ is an anti-involution as in the definition of a weakly triangular decomposition, then $\h \oplus {\rm span}_\cc \{ X_\alpha, \sigma ( X_\alpha ) \}$ and $\h \oplus {\rm span}_\cc \{ X_\beta, \sigma (X_\beta ) \}$ are both subalgebras of $\LL$ of dimension $\dim \h + 2$.  Since $\alpha \neq \beta$, the intersection of these subalgebras is $\h$.  It follows that $\mathsf T = \h$.  Since $\tau$ must permute subalgebras of dimension $\dim \h + 2$, it follows that $\tau (\h ) \subseteq \h$.

 To see that $\tau ( \LL_\alpha ) \subseteq \LL_{\alpha \circ \tau^{-1}}$, let $x \in \LL_\alpha$ and $h \in \h$.  Then 
\begin{align*}
[h, \tau (x)] &= \tau ( \tau^{-1} ( [h, \tau (x)] )) = \tau ( [ \tau^{-1}(h), x ] ) = \tau (  (\alpha \circ \tau^{-1} (h)) x) \\
&= (\alpha \circ \tau^{-1} (h)) \tau (x)
\end{align*}
\end{proof}

Applying Proposition \ref{prop:autPreserveCartan} to the insertion-elimination Lie algebra gives the following.

\begin{cor}\label{cor:autPreserveCartanIE}
If $\tau$ is an automorphism of the insertion-elimination Lie algebra $\g$, then $0 \neq \tau (d) \in \cc d$.  
\end{cor}

\begin{lem}\label{lem:autWeights}
Let $\g$ be the insertion-elimination algebra, and let $\tau \in \aut_\cc (\g)$.  Then either
\begin{itemize}
\item $\tau (d) = d$ and $\tau (D_t^\pm) \in \g_{\pm |t|}$ for all $t \in \T$; or
\item $\tau (d) = -d$ and $\tau (D_t^\pm) \in \g_{\mp |t|}$ for all $t \in \T$.
\end{itemize}
\end{lem}
\begin{proof}
From Corollary \ref{cor:autPreserveCartanIE}, we know that $\tau (d) = \mu d$ for some $\mu \in \cc^*$.   The containment $\tau ( \LL_\alpha ) \subseteq \LL_{\alpha \circ \tau^{-1}}$ of Proposition \ref{prop:autPreserveCartan}, implies that $\tau ( \g_k ) \subseteq \g_{k / \mu}$ for all $k \in \z$.  This forces $\frac{k}{\mu} \in \z$ for all $k \in \z$, so $\frac{1}{\mu} \in \z$.  Write $m = \frac{1}{\mu} \in \z$.  Then since $\tau ( \g_k ) \subseteq \g_{mk}$, we have 
$$
\g = \tau(\g) =\bigoplus_{k \in \z} \tau (\g_k) \subseteq \bigoplus_{k \in \z} \g_{mk}.
$$
This forces $|m|=1$, and thus $\mu \in \{ 1, -1 \}$.  The containment $\tau ( \g_k ) \subseteq \g_{k/ \mu}$ now proves the assertion regarding $\tau (D_t^{\pm})$.
\end{proof}

\begin{lem}\label{lem:autoFix_d-Size1}
Let $\omega \in \aut_\cc (\g)$ such that $\omega (d) = d$ and $\omega ( D_{\pstree{\Tr{$\bullet$}}{}}^+ ) = D_{\pstree{\Tr{$\bullet$}}{}}^+$.  Then $\omega = {\rm id}_\g$.
\end{lem}

\begin{proof}
We prove $\omega(D_t^\pm)=D_t^\pm$ by induction on $|t|$.  
By Lemma  \ref{lem:autWeights}, $\omega (D^-_{\pstree{\Tr{$\bullet$}}{}}) = c D^-_{\pstree{\Tr{$\bullet$}}{}}$ for some $c \in \cc$.  Applying $\omega$ to the equation $d= [D^-_{\pstree{\Tr{$\bullet$}}{}}, D^+_{\pstree{\Tr{$\bullet$}}{}}]$ yields $c=1$.  This completes the base case.

Now, for a fixed $n > 1$ we first consider $t \in \T$ with $|t| = n$ and $\rootdeg (t) = 1$. By Lemma \ref{lem:autWeights}, we may write 
\begin{equation} \label{eqn:coeffD_t=1}
\omega (D_t^+) = \sum_{u \in \T, |u| = n} c_u D_u^+.
\end{equation}
for some $c_u \in \cc$.
Since $\rootdeg (t) = 1$, we have $t = \bullet \cup_{rt(\bullet)} \tilde t$ for some $\tilde t \in \T$.   By induction, $\omega ( D_{\tilde t}^-) = D_{\tilde t}^-$ since $| \tilde t| = n-1$.  Then $\psset{levelsep=0.3cm, treesep=0.3cm}
[ D_{\tilde t}^-, D_t^+ ] = D_{\pstree{\Tr{$\bullet$}}{}}^+$; if we apply $\omega$ to this equation, we obtain 
$$\psset{levelsep=0.3cm, treesep=0.3cm}
D_{\pstree{\Tr{$\bullet$}}{}}^+ 
= [ D_{\tilde t}^-, \sum_{u \in \T, |u| = n} c_u D_u^+ ] \\
=c_t \psset{levelsep=0.3cm, treesep=0.3cm}
[ D_{\tilde t}^-, D_t^+ ] = c_t D_{\pstree{\Tr{$\bullet$}}{}}^+.
$$
In particular, $[ D_{\tilde t}^-, D_u^+ ] = 0$ for $u \neq t$ with $|u| = n$, since $t$ is the unique tree of size $n$ with $\tilde t$ as a component.  Therefore, $c_t = 1$.

Suppose that there exists $u \neq t$ with $c_u \neq 0$ in (\ref{eqn:coeffD_t=1}).  Among those $u \neq t$ with $c_u \neq 0$, fix $q$ with $\comps (q)$ maximal.  Let $q_0$ be a component of $q$ with $|q_0|=\comps (q)$, and regard $q  = q' \cup_v q_0$, where $v = \rt (q') = \rt (q)$.  Let $X = [D_{q_0}^-, D_t ^+] \in {\rm span}_\cc \{ D_s^+ \mid |s| < n \}$.  Then $\omega (X) = X$ and $\omega (D_{q_0}^-) = D_{q_0}^-$ by induction, so applying $\omega$ to the equation $X = [D_{q_0}^-, D_t ^+]$ yields 
\begin{align*}
X &= [D_{q_0}^-, \omega (D_t ^+) ] = [D_{q_0}^-, D_t^+ + \sum_{t \neq u \in \T, \; |u| = n} c_u D_u^+ ] \\
&= X + [D_{q_0}^-, \sum_{t \neq u \in \T, \; |u| = n} c_u D_u^+ ] \\
&= X + \sum_{t \neq u \in \T, \; |u| = n} c_u [D_{q_0}^-,  D_u^+ ].\\
\Rightarrow \quad 0&=\sum_{t \neq u \in \T, \; |u| = n} c_u [D_{q_0}^-,  D_u^+ ]. 
\end{align*}
By our choice of $q$, we have that $c_u=0$ if $\comps(u)>|q_0|$.  
 If $\comps (u)<|q_0|$, then $[D_{q_0}^-,  D_u^+ ] = 0$.  If $\comps (u)=| q_0 |$, then $[D_{q_0}^-,  D_u^+ ] = k_s D_{s}^+$, where $k$ is the number of components (possibly 0) of $u$ that are isomorphic to $q_0$ and $s$ is the tree formed from $u$ by removing some (any) component isomorphic to $q_0$.   In particular, $[D_{q_0}^-,  D_{q}^+ ] = k_{q'} D_{q'}^+$; and the coefficient of $D_{q'}^+$ is 0 in $[D_{q_0}^-,  D_u^+]$ when $u \neq q$ with $c_u \neq 0$.   This forces $c_{q} = 0$, which is a contradiction.  Thus the assumption that $c_u \neq 0$ for some $u \neq t$ must be incorrect.

We now have that $\omega (D_t^+) = D_t^+$ whenever $|t| = n$ and $\rootdeg (t) = 1$, and we also have $\omega (D_r^+) = D_r^+$ whenever $|r| < n$.  From Theorem \ref{thm:generators} (below), we know that every element of $\g_n \subseteq \g^+$ can be generated by elements $D_r^+ \in \bigoplus_{1 \le i \le n} \g_i$ such that $\rootdeg (r) = 1$.  Thus if $u \in \T$ with $|u| = n$ (and not necessarily $\rootdeg (u) = 1$), $D_u^+$ can be expressed in terms of elements of $\g^+$ that are fixed by $\omega$, so $\omega (D_u^+) = D_u^+$.

Finally, we show that $\omega (D_t^-) = D_t^-$ for any $t \in \T$.  Assume $|t| = n$.  From Lemma \ref{lem:autWeights}, we have $\omega (D_t^-) = \sum_{u \in \T, |u| = n} b_u D_u^-$ for $b_u \in \cc$.   If $s \in \T$ with $|s| = n$, then $[D_t^-, D_s^+] = \delta_{t,s} \, d$ and $\omega (D_s^+) = D_s^+$.  Thus we have 
$$\delta_{t,s} \, d = \omega ( [D_t^-, D_s^+] ) = [ \omega (D_t^-), \omega (D_s^+)] = [ \sum_{u \in \T, |u| = n} b_u D_u^- , D_s^+ ] = b_s d.$$
Thus $b_t = 1$ and $b_s = 0$ for $s \neq t$.  
\end{proof}

For $\zeta \in \cc^* = \cc \setminus \{ 0 \}$, define the linear map $\tau_{\zeta}: \g \rightarrow \g$ defined by
\begin{equation}\label{eqn:tau_zeta}
\tau_\zeta (d) = d, \qquad \tau_\zeta (D_t^+) = \zeta^{|t|} D_t^+, \qquad \tau_\zeta (D_t^-) = \zeta^{-|t|} D_t^-, \quad \mbox{for $t \in \T$}
\end{equation}
It's straightforward to verify that $\tau_{\zeta}$ is an automorphism of $\g$.  (We note the automorphisms $\tau_\zeta$ are of the form  $\tau_\theta (x_n) = \theta(n) x_n$ where $\theta: \z \rightarrow cc^*$ is a homomorphism, $\z$ grades $\g$, and $x_n \in \g_n$.)

\begin{lem}\label{lem:tauFix_d}
Let $\tau \in \aut_\cc (\g)$ such that $\tau (d) = d$.  Then $\tau = \tau_\zeta$ for some $\zeta \in \cc^*$.  
\end{lem}
\begin{proof}
Since $\tau (d) = d$, Lemma \ref{lem:autWeights} implies that $\tau (\g_m) \subseteq \g_m$ for all $m \in \z$.  Let $\zeta \in \cc^*$ such that $\tau (D_{\pstree{\Tr{$\bullet$}}{}}^+) = \zeta D_{\pstree{\Tr{$\bullet$}}{}}^+$, and define $\tau_\zeta$ as in (\ref{eqn:tau_zeta}).  For $\omega = \tau \circ \tau_{\zeta}^{-1}$, we get that $\omega (d) = d$, $\omega ( D_{\pstree{\Tr{$\bullet$}}{}}^+ ) = D_{\pstree{\Tr{$\bullet$}}{}}^+$.    The result now follows from Lemma \ref{lem:autoFix_d-Size1}.  
\end{proof}

Define an anti-automorphism $S: \g \rightarrow \g$ by $S(x)=-x$ for all $x \in \g$.  (Extended to $U(\g)$, this is the standard Hopf algebra antipode.) With $\sigma : \g \to \g$ as in Theorem \ref{prop:antiAutoSymm}, let 
\begin{equation}\label{eqn:tau0}
\tau_0  = \sigma \circ S.
\end{equation}
Since the composition of two anti-automorphisms is an automorphism, we have that $\tau_0 \in \aut_\cc (\g)$.

We can now describe $\aut_\cc (\g)$.

\begin{thm}\label{thm:AutInsElAlg}
Let $\mathcal A_{\cc^*} = \{ \tau_\zeta \mid \zeta \in \cc^* \}$, and $\mathcal A_0 = \{ {\rm id}_\g , \tau_0 \}$.  Then 
$$\aut_\cc (\g) =  \mathcal A_{\cc^*} \rtimes \mathcal A_0,$$
where $\tau_0 \circ \tau_\zeta \circ \tau_0 = \tau_{\zeta^{-1}}$.
\end{thm}

\begin{proof}
Clearly both $\mathcal A_{\cc^*}$ and $\mathcal A_0$ are subgroups of $\aut_\cc (\g)$. (In particular, $\tau_\zeta \circ \tau_\nu = \tau_{\zeta \nu}$ for $\zeta, \nu \in \cc^*$ and $\tau_0^2 = {\rm id}_\g$.)

We first show that for $\tau \in \aut_\cc (\g)$, there exist $\tau_\zeta \in \mathcal A_{\cc^*}$ and $\gamma \in \mathcal A_0$ such that $\tau = \tau_\zeta \circ \gamma$.  
Lemma \ref{lem:autWeights} implies that either $\tau (d) = d$ or $\tau (d) = -d$.  If $\tau (d) = d$, then from Lemma \ref{lem:tauFix_d} we have that $\tau = \tau_\zeta$ for some $\zeta \in \cc^*$. Then the result follows with $\gamma = {\rm id}_\g$.  If $\tau (d) = -d$, then $(\tau \circ \tau_0 )(d) = d$; and thus Lemma \ref{lem:tauFix_d} implies that $\tau \circ \tau_0 = \tau_\zeta$ for some $\zeta \in \cc^*$.  Therefore $\tau = \tau_\zeta \circ \tau_0$ as desired.  The uniqueness of the decomposition $\tau = \tau_\zeta \circ \gamma$ follows from the fact that $\mathcal A_{\cc^*} \cap \mathcal A_0 = \{ {\rm id}_\g \}$.

If $t \in \T$, it is straightforward to verify that $(\tau_0 \circ \tau_\zeta \circ \tau_0) (D_t^+) = \tau_{\zeta^{-1}} ( D_t^+)$ and $(\tau_0 \circ \tau_\zeta \circ \tau_0) (D_t^-) = \tau_{\zeta^{-1}} (D_t^-)$.  Since $(\tau_0 \circ \tau_\zeta \circ \tau_0) (d) = d = \tau_{\zeta^{-1}}(d)$, the claim $\tau_0 \circ \tau_\zeta \circ \tau_0 = \tau_{\zeta^{-1}}$ follows. \end{proof}

\begin{cor}\label{cor:centerAutIE}
The center $Z(\aut_\cc (\g))$ of $\aut_\cc (\g)$ is $\mathcal A_0= \{ {\rm id}_\g, \tau_{-1} \}$.
\end{cor}
\begin{proof}
It is clear that $\{ {\rm id}_\g, \tau_{-1} \} \subseteq Z(\aut_\cc (\g))$.  To show the opposite containment, consider an arbitrary element $\tau_\zeta \circ \gamma \in \aut_\cc (\g)$, where $\zeta \in \cc^*$ and $\gamma \in \mathcal A_0$.  If $\tau_\zeta \circ \gamma$ is central, then it is fixed under conjugation by $\tau_0 = \tau_0^{-1}$; therefore, 
$$\tau_\zeta \circ \gamma = \tau_0 \circ \tau_\zeta \circ \gamma \circ \tau_0 = \tau_0 \circ \tau_\zeta \circ \tau_0 \circ \gamma = \tau_{\zeta^{-1}} \circ \gamma.$$
This implies that $\tau_\zeta = \tau_{\zeta^{-1}}$, and thus $\zeta = \pm 1$. Moreover, this shows that if $\tau_\zeta \circ \gamma$ is central, then $\tau_\zeta$ is central.  Consequently $\gamma \in \mathcal A_0$ must be central.  Since $\tau_0$ is not central, we have $\gamma = {\rm id}_\g$.
\end{proof}

\begin{cor}
Every anti-automorphism of $\g$ has the form $\sigma \circ \tau$, with $\tau \in \aut_\cc (\g) = \mathcal A_{\cc^*} \rtimes \mathcal A_0$.
\end{cor}

\begin{exam}\label{exam:genVirAut}
Let $V(M)$ denote the generalized Virasoro algebra as described in Section \ref{subsec:genVir}.  For $\theta : M \to \mathbb F^*$ with $\theta ( \alpha + \beta ) = \theta (\alpha ) \theta (\beta)$, define $\tau_\theta \in \aut_\cc (\g)$ by 
$$\tau_\theta (z) = z \qquad \mbox{and} \qquad \tau_\theta (e_\alpha ) = \theta (\alpha ) e_\alpha.$$
If $\zeta \in \mathbb F^*$ with $M \zeta^{-1} = M$, define $\kappa_\zeta : V(M) \to V(M)$ by $\kappa_\zeta (e_\alpha ) = \zeta e_{\alpha \zeta^{-1}}$ and $\kappa_\zeta (z) = \zeta^{-1}z$.  In \cite{SZ} Theorem 2.3, it is shown that 
$${\rm Aut}_\F (V(M)) = \mathcal A_0 \ltimes \mathcal A_1,$$
where $\mathcal A_0 = \{ \kappa_\zeta \mid \zeta \in \mathbb F^*, \ M \zeta^{-1} = M \}$ and $\mathcal A_1 = \{ \tau_\theta \mid \theta \in {\rm Hom}(M, \mathbb F^*) \}$.

We note that this closely mirrors the structure of the automorphism group of the insertion-elimination algebra and is consistent with the general structure described in Proposition \ref{prop:autPreserveCartan}. 

\end{exam}

\section{Derivations for graded Lie algebras}

Let $\LL$ be a Lie algebra over a field $\mathbb F$, and suppose $V$ is an $\LL$-module with $\LL$-action denoted by $x.v$ for $x \in \LL$ and $v \in V$.  A {\it derivation} from $\LL$ to $V$ is a linear map $\phi: \LL \rightarrow V$  such that 
$$\phi([x,y])=x. \phi(y)-y. \phi(x)$$
for all $x,y \in \LL$.  A derivation $\phi : \LL \to V$ is said to be \emph{inner} if there exists $v \in V$ such that $\theta (x) = x.v$ for all $x \in \LL$.  In the case that $V = \LL$, then $\phi : \LL \to \LL$ is inner if there exists $y \in \LL$ such that $\phi (x) = [x,y]$.   We let $\der_\F ( \LL , V)$ denote the space of all derivations from $\LL$ to $V$.  The Lie algebra of derivations from $\LL$ to $\LL$ and its ideal of inner derivations  are denoted $\der_\F (\LL)$ and $\inn_\F (\LL)$,  respectively.

Suppose $\LL$ and $V$ are graded by a group $A$.  A derivation $\phi \in \der_\F (\LL, V)$ has {\it degree $g$}, where $g \in A$, if $\phi (\LL_h) \subseteq V_{g+h}$ for all $h \in A$.  Let $\der_\F (\LL, V)_g$ denote the subspace of degree $g$ derivations from $\LL$ to $V$. Define $\inn_\F (\LL, V)_g=\der_\F (\LL, V)_g \cap \inn (\LL, V)$.

The following result is comparable to Proposition 1.2 in \cite{Farn}, where we replace the assumption that $\LL$ is finitely generated with the assumption that $\LL^{\pm}$ are completely self-centralizing. (Portions of the proof here closely follow the arguments in \cite{Farn}.) \cite{DZ} also present a similarly general result but use {\it locally inner} derivations to describe $\der_\F (\LL)$.

\begin{prop} \label{prop:DerDirectSum}
Let $\LL = \bigoplus_{\alpha \in G} \LL_\alpha$ be a Lie algebra admitting a regular weakly triangular decomposition.  Assume that $\LL^+=\bigoplus_{\alpha \in G_+} \LL_{\alpha}$ is completely self-centralizing.  Then
$$\der_\F (\LL) = \der_\F (\LL)_0 + \inn_\F (\LL).$$
\end{prop}

\begin{proof}
In this proof, we view $\LL$ as graded over the group $\h^*$, where $\LL_{\alpha}=0$ for $\alpha \not\in G = G_- \cup \{ 0 \} \cup G_+$.

Let $\phi \in \der_\F (\LL)$. For $\alpha \in \h^*$, let $\proj_\alpha: \LL \rightarrow \LL_\alpha$ denote the canonical projection.  Define $\phi_\alpha$ by 
\begin{equation} \label{eqn:def-phi_alpha}
\phi_\alpha(x)=\sum_{\beta \in \h^*} \proj_{\alpha+\beta} \circ \phi \circ \proj_\beta (x)
\end{equation}
for $x \in \LL$.   Note that this sum contains only finitely many nonzero terms and so is well-defined.

We first show that $\phi =\prod_{\alpha \in \h^*} \phi_\alpha$.   It is straightforward to verify that $\phi_\alpha \in \der_\F (\LL)$. Let $x \in \LL$.  Then,
\begin{align*}
\phi(x) 
&= \sum_{\beta \in \h^*} (\sum_{\alpha \in \h^*} \proj_{\alpha+\beta} \circ \phi \circ \proj_\beta (x) ) \\
&=  \sum_{\alpha \in \h^*}  (\sum_{\beta \in \h^*} \proj_{\alpha+\beta} \circ \phi \circ \proj_\beta (x) ) \\
&= \sum_{\alpha \in \h^*} \phi_\alpha (x).
\end{align*}
Note that the order of the summations may be switched for a fixed $x$ since there are only finitely many non-zero summands in each sum. 

Next we show that $\phi_\alpha \in  \inn_\F (\LL)$ for every $\phi \in \der_\F (\LL)$ and $\alpha \neq 0$.  Note that $\phi_\alpha |_{\h}$ is a derivation from $\h$ to $\LL_\alpha$.   Since $\dim \h, \dim \LL_\alpha < \infty$, Lemma 3 of \cite{Barnes} implies that $H^1(\h, \LL_\alpha) =0$; therefore, $\phi_\alpha |_{\h}$  must be an inner derivation.  Let $y_{\alpha} \in \LL_\alpha$ be such that $\phi_\alpha(h) = [y_\alpha, h]$ for $h \in \h$. Define a new (degree $\alpha$) derivation $\phi_\alpha': \LL \rightarrow \LL$ by $\phi_\alpha'(h)=\phi_\alpha(h) -[y_\alpha, h]$. Since $\phi_\alpha' \mid_{\h}=0$ and $\phi_\alpha' : \LL \to \LL$ is a derivation, we see that, for $h \in \h$ and $x \in \LL_\beta$ (with $\beta \in \h^*$), $\phi_\alpha'([h, x]) = [ \phi_\alpha'(h), x] + [ h, \phi_\alpha'(x)] = [h, \phi_\alpha'(x)]$; it follows that $\phi_\alpha' : \LL_\beta \to \LL_{\alpha + \beta}$ is an $\h$-homomorphism.  However, a straightforward weight space argument shows that $\Hom_{\h} (\LL_\beta, \LL_{\alpha+\beta})=0$ if $\beta \neq 0$. Therefore, for $\beta \neq 0$, we get that $\phi_\alpha' \mid_{\LL_\beta} =0$. As we now have that $\phi_\alpha' = 0$ on all of $\LL$, it follows that $\phi_\alpha$ is inner.

For $\phi \in \der_\F (\LL)$, we now know that $\phi = \sum_{\alpha \in \h^*} \phi_\alpha$, where $\phi_\alpha \in \der_\F (\LL)_\alpha = \inn_\F (\LL)_\alpha$ whenever $\alpha \neq 0$.  Thus $\der_\F (\LL) = \der_\F (\LL)_0 + \prod_{\alpha \in \h^*} \inn_\F (\LL)_\alpha$.   To complete the proof, it is enough to show that $\phi_\alpha = 0$ for all but finitely many $\alpha \in \h^*$. Since $\phi_\alpha \in \inn_\F (\LL)_\alpha$ whenever $\alpha \neq 0$, we may write $\phi_\alpha=\ad_{X_\alpha}$, where $X_\alpha \in \LL_\alpha$ and $0 \neq \alpha \in G = G_- \cup \{ 0 \} \cup G_+ \subseteq \h^*$.   We show that $X_\alpha = 0$ (and thus $\phi_\alpha = 0$) for all but finitely many $\alpha \in G$.

For $\beta \in G_+$ and $0 \neq y \in \LL_\beta$,
$$\phi(y) = \sum_{\alpha \in \h^*} \phi_\alpha(y) = \phi_0(y) + \sum_{0 \neq \alpha \in G} [X_\alpha, y].$$
Since $\LL^+$ is completely self-centralizing, we know that $[X_\alpha, y]$ is a nonzero element of $\LL_{\alpha + \beta}$ whenever $X_\alpha \neq 0$, $\beta \neq \alpha$, and $\alpha \in G_+$.  Therefore, for $\phi$ to be well-defined, the sum above must contain only finitely many terms and thus $X_\alpha = 0$ for all but finitely many $\alpha$ in $G_+$ (and similarly $G_-$).  
\end{proof}

\begin{lem}\label{lem:deg0DerInsElm}
Let $\g$ be the insertion-elimination Lie algebra, and let $\delta \in \der_\cc (\g)$ have degree 0.  Suppose that $\delta (d) = 0$ and $\psset{levelsep=0.3cm, treesep=0.3cm}
\delta (D_{\pstree{\Tr{$\bullet$}}{}}^{\pm}) = 0$.  Then $\delta = 0 \in \der_\cc (\g)$.  
\end{lem}
\begin{proof}
We prove by induction on $|t|$ that $\delta (D_t^{\pm}) = 0$ for all $t \in \T$.  The base case is part of the stated assumptions.  Fix $n>1$ and $t \in \T_n$, and write 
$$\delta (D_t^+) = \sum_{s \in \T_n} c_s D_s^+.$$

Let $\mathcal S = \{ s \in \T_n \mid c_s \neq 0 \}$, and suppose $\mathcal S \neq \emptyset$.  Choose $u \in \mathcal S$ with $\comps (u)$ (that is, the maximal size among all components of $u$) maximal among elements of $\mathcal S$.  Let $u_0$ be a component of $u$ with $|u_0|=\comps (u)$; then $|u_0| < |u| = n$ and $\displaystyle{[D_{u_0}^-, D_t^+] \in \bigoplus_{0 \le i < n} \g_i}$.  It then follows from induction that $\delta ([D_{u_0}^-, D_t^+]) = 0$ and $\delta (D_{u_0}^-) = 0$. Thus
$$0 = \delta ([D_{u_0}^-, D_t^+]) = [ \delta (D_{u_0}^-) , D_t^+] + [D_{u_0}^-, \delta (D_t^+)] =  [D_{u_0}^-, \delta (D_t^+)].$$
From our choice of $u \in \mathcal S$, we can rewrite this as 
$$0 = [D_{u_0}^-, \sum_{s \in \T_n} c_s D_s^+ ] = \sum_{s \in \T_n \mid \  \comps (s) \le |u_0|} c_s  [D_{u_0}^-, D_s^+ ].$$
In the above sum, since $\comps (s) \le |u_0|$, it follows that $[D_{u_0}^-, D_s^+] = k_s D_{\tilde s}^+$, where $k_s$ is the number of components of $s$ (possibly 0) isomorphic to $u_0$ and $\tilde s$ is the unique (up to isomorphism) rooted tree formed from $s$ by removing any component of $s$ isomorphic to $u_0$.  In particular, we have $[D_{u_0}^-, D_u^+]$ is a nonzero multiple of $D_{\tilde u}^+$.  Moreover, if for some $s \in \T_n$ with $\comps (s) \le |u_0|$, we have $\tilde s  = \tilde u$, then $s  = u$.  This leads to the contradiction that the coefficient of the term $D_{\tilde u}^+$ in the above sum is $c_u k_u \neq 0$.  Therefore the assumption that $\mathcal S \neq \emptyset$ must be incorrect.  

It follows that $\delta (D_t^-) = 0$ by a similar argument.
\end{proof}

\begin{cor}\label{cor:DerModInIE}
Let $\g$ be the insertion-elimination algebra.  Then  $\der_\cc (\g) = \inn_\cc  (\g)$, and thus $H^1(\g, \g) \cong \der_\cc (\g) / \inn_\cc ( \g )$ is trivial.
\end{cor}
\begin{proof}
By Proposition \ref{prop:DerDirectSum}, it suffices to show that $\der_\cc (\g)_0 \subseteq \inn_\cc (\g)$.   Let $\theta \in \der_\cc (\g)_0$.  Then we have $\theta (d) = cd$ and $\psset{levelsep=0.3cm, treesep=0.3cm} \theta (D_{\pstree{\Tr{$\bullet$}}{}}^+) = k D_{\pstree{\Tr{$\bullet$}}{}}^+$ for $c, k \in \cc$.  By applying $\theta$ to the relation $\psset{levelsep=0.3cm, treesep=0.3cm} [d, D_{\pstree{\Tr{$\bullet$}}{}}^+] = D_{\pstree{\Tr{$\bullet$}}{}}^+$, it is easy to see that $c = 0$, so we now have $\theta (d) = 0$ and $\psset{levelsep=0.3cm, treesep=0.3cm} \theta (D_{\pstree{\Tr{$\bullet$}}{}}^+) = k D_{\pstree{\Tr{$\bullet$}}{}}^+$.  It is also straightforward to use the relation $\psset{levelsep=0.3cm, treesep=0.3cm} [ D_{\pstree{\Tr{$\bullet$}}{}}^-, D_{\pstree{\Tr{$\bullet$}}{}}^+]$ to show that $\psset{levelsep=0.3cm, treesep=0.3cm} \theta (D_{\pstree{\Tr{$\bullet$}}{}}^-) = -k D_{\pstree{\Tr{$\bullet$}}{}}^-$.  From this, we see that the derivation $\delta := \theta - \ad_{kd}$ is $0$ by Lemma \ref{lem:deg0DerInsElm}, so $\theta = \ad_{kd} \in \inn_\cc (\g)$.
\end{proof}

\begin{exam}
Let $M \subseteq \F$ be an additive subgroup of $\F$ with a additive total order and $V(M)$ the generalized Virasoro algebra for $M$, described in Section \ref{subsec:genVir}.  Since $V(M)^{\pm}$ are completely self-centralizing, Proposition \ref{prop:DerDirectSum} applies.  Thus to determine $\der_\F (V(M))$, it suffices to determine $\der_\F  (V(M))_0$.  From the defining relations of $V(M)$, it is straightforward to show that $\der_\F  (V(M))_0$ can be identified with $\Hom(M, \F)$.  Namely, for $\theta: M \rightarrow \F$, there is a degree 0 derivation $\delta_\theta$ given by 
$$\delta_\theta (e_\alpha ) = \theta ( \alpha) e_\alpha \ \ ( \alpha \in M), \qquad \delta_\theta (z) = 0.$$
Moreover, $\delta_\theta \in \inn_\F (V(M))$ if and only if $\theta \in \F  \, {\rm id}_M$.  

These results are presented for a more general class of (generalized) Virasoro algebras in Proposition 3.3 and Theorem 3.4 of \cite{DZ} using the idea of locally inner derivations (rather than completely self-centralizing subalgebras).
\end{exam}

\section{Generating the insertion-elimination algebra}

In this section, we first construct a generating set for the insertion-elimination algebra; this result is used in the proof of Lemma \ref{lem:autoFix_d-Size1}.  The infinite generating set that we construct is not minimal; however, we do show that this Lie algebra is not finitely generated.

\subsection{A generating set for the insertion-elimination algebra.}

For a fixed $n$, define a partial order $\prec$ on $\T_n$ as follows: for $s, t \in \mathbb T_n$, $s \prec t$ if
\begin{enumerate}
\item $\depth (s) > \depth (t)$; or 
\item $\depth (s) = \depth (t)$ and $\rootdeg (s) < \rootdeg (t)$.
\end{enumerate}

\begin{lem} \label{lem:GeneratingSet}
Let $t \in \T_n$ such that $\rootdeg (t)>1$.  Then $D_t^+$ can be written as a linear combination of 
\begin{itemize}
\item a (single) commutator $[D_s^+, D_{s'}^+]$, for some $s, s' \in \T$ with $|s|,|s'|<n$; 
\item elements $D_u^+$, where $u \in \T_n$ with $u \prec t$. 
\end{itemize}
\end{lem}
\begin{proof}
Let $t_1, \ldots, t_k$ be the components of $t$ (so that $t= \bigcup_i t_i$) and assume (without loss of generality) that $\depth (t_1) \leq \cdots \leq \depth (t_k)$. Define $s'=\bigcup_{i=2}^k t_i$, and note that $\depth (s')= \depth (t)$ and $\rootdeg(s')<\rootdeg(t)$.

Then,
$$[D_{t_1}^+, D_{s'}^+] = D_t^+ + \sum_{t \neq u \in \T_n} (\beta(u, t_1, s')-\beta(u, s', t_1)) D_u^+,$$
We then have the following:
\begin{itemize}
\item If $\beta(u, t_1, s') \neq 0$, the tree $u$ has the form $u=s' \cup_v t_1$ for some $v \in V(s')$ such that $v \neq \rt (s')$. Then, $\depth (u) \geq \depth (s')= \depth (t)$ and $\rootdeg(u)=\rootdeg (s')<\rootdeg (t)$.  Thus $u \prec t$.
\item If $\beta(u, s', t_1) \neq 0$, the tree $u$ has the form $u= t_1 \cup_v s'$ for some $v \in V(t_1)$.   Then $\depth (u)> \depth (s')= \depth (t)$ and so $u \prec t$. 
\end{itemize}
Therefore, the statement of the lemma follows with $s=t_1$.
\end{proof}

This lemma allows us to prove Theorem \ref{thm:generators} below, which is used to prove Lemma \ref{lem:autoFix_d-Size1}.

\begin{thm}\label{thm:generators}
The set $\mathbb B= \{ D_t^+ \mid t \in \mathbb T,  \rootdeg(t)=1 \}$ is a generating set for $\n^+$.
\end{thm}

\begin{proof}
We must show that for any $t \in \mathbb T$, $D_t^+$ can be written as a linear combination of elements of $\mathbb B$ and (potentially nested) commutators of elements of $\mathbb B$. We prove this by inducting on $|t|$, noting that the case $|t|=1$ is trivially true.

Therefore, let $n>1$ and suppose $t \in \mathbb T_n$ with $\rootdeg(t)>1$.  Lemma \ref{lem:GeneratingSet} implies that $D_t^+$ is a linear combination of a commutator $[D_s^+, D_{s'}^+]$ ($|s|,|s'|<n$) and elements $D_{t'}^+$ where $|t'|=n$ but $t' \prec t$. Using the inductive hypothesis, the commutator $[D_s^+, D_{s'}^+]$ has the appropriate form.  Therefore, we only need to consider $D_{t'}^+$, where $t' \prec t$.  If $\depth (t') = n-1$, then $\rootdeg(t')=1$, and we are done.  Otherwise, we may again apply Lemma  \ref{lem:GeneratingSet} to $D_{t'}^+$.  Since $0 \le \depth (t') \leq n-1$ and  $1 \le \rootdeg (t') \le n-1$ for all $t' \in \mathbb T_n$, this process only needs to be repeated a finite number of times.
\end{proof}

We note that the set in Theorem \ref{thm:generators} is not minimal.  In particular, the proper subset 
$$\psset{levelsep=0.3cm, treesep=0.3cm}
\mathbb B \setminus \{ D^+_{\pstree{\Tr{$\bullet$}}{\pstree{\Tr{$\bullet$}}{ \Tr{$\bullet$} \Tr{$\bullet$} \Tr{$\bullet$} }}} \}$$
generates $\n^+$, which is a consequence of the following calculation:
$$
 \psset{levelsep=0.3cm, treesep=0.3cm}
-2 [D^+_{\pstree{\Tr{$\bullet$}}{}}, D^+_{\pstree{\Tr{$\bullet$}}{\pstree{\Tr{$\bullet$}}{ \Tr{$\bullet$}  \Tr{$\bullet$} }}}]  
+ [D^+_{\pstree{\Tr{$\bullet$}}{}}, D^+_{\pstree{\Tr{$\bullet$}}{\pstree{\Tr{$\bullet$}}{ \Tr{$\bullet$} }  \Tr{$\bullet$} }}] 
- [D^+_{\pstree{\Tr{$\bullet$}}{\Tr{$\bullet$}}}, D^+_{\pstree{\Tr{$\bullet$}}{\pstree{\Tr{$\bullet$}}{ \Tr{$\bullet$} }}}] 
- [D^+_{\pstree{\Tr{$\bullet$}}{\Tr{$\bullet$}}}, D^+_{\pstree{\Tr{$\bullet$}}{ \Tr{$\bullet$} \Tr{$\bullet$} }}] 
$$
$$\psset{levelsep=0.3cm, treesep=0.3cm}
= 3 D^+_{\pstree{\Tr{$\bullet$}}{\pstree{\Tr{$\bullet$}}{ \pstree{\Tr{$\bullet$}}{ \Tr{$\bullet$} \Tr{$\bullet$} }}}} 
- 6 D^+_{\pstree{\Tr{$\bullet$}}{\pstree{\Tr{$\bullet$}}{ \pstree{\Tr{$\bullet$}}{\Tr{$\bullet$}}  \Tr{$\bullet$} }}} 
- 2 D^+_{\pstree{\Tr{$\bullet$}}{\pstree{\Tr{$\bullet$}}{ \Tr{$\bullet$} \Tr{$\bullet$} \Tr{$\bullet$} }}}.
$$

\subsection{The insertion-elimination algebra is not finitely generated.}
We show in this section that the insertion-elimination algebra $\g = \g^- \oplus \h \oplus \g^+$ is not finitely generated as a Lie algebra.  A key step is to relate the problem of generating $\g$ to the problem of generating $\g^+$ from a subset of $\g^+$.

Lemma \ref{lem:reorderBrackets}, Lemma \ref{lem:prodReducePos}, and Lemma \ref{lem:finGenPosPart} below can be proved in the more general setting of a $\z$-graded Lie algebra.  However, for simplicity we state the results in terms of the insertion-elimination Lie algebra.

For $0 < M \in \z$, we let 
$$\g^+_{[M]} = \bigoplus_{0 \le i \le M} \g_i, \quad \g^-_{[M]} = \bigoplus_{-M \le i \le 0} \g_i, \quad \mbox{and} \quad \g_{[M]} = \bigoplus_{-M \le i \le M} \g_i.$$
If $S \subseteq \g$ is any subset, let $\langle S \rangle$ denote the Lie subalgebra of $\g$ generated by $S$.  In particular, $\langle \g^+_{[M]} \rangle$ denotes the Lie subalgebra of $\g$ generated by $\g^+_{[M]}$.

\begin{lem}\label{lem:reorderBrackets}
Let $S$ be a subset of the insertion-elimination Lie algebra $\g$.  Then every element of the Lie subalgebra $\langle S \rangle$ of $\g$ generated by $S$ can be written as a sum of elements of the form 
$$( \ad_{x_1} \circ \ad_{x_2} \circ \cdots \circ \ad_{x_{k-1}})(x_k)=[x_1, [x_2, \ldots [ x_{k-1}, x_k ] \ldots ]],$$
for $x_1, \ldots, x_k \in S$.
\end{lem}
\begin{proof}
This follows from the corresponding statement about free Lie algebras.  Thus the result is a consequence of Lemma 3.3 of \cite{Ch}.
\end{proof}

\begin{lem}\label{lem:prodReducePos}
Let $x_1, \ldots, x_k \in \g$ (the insertion-elimination Lie algebra) and $n_1, \ldots, n_k \in \z$ such that $x_i \in \g_{n_i}$, and let $M = \max \{ |n_1|, \ldots, |n_k| \}$.  If $n_1 + \cdots + n_k \ge 0$, then $[x_1, [x_2, \ldots [ x_{k-1}, x_k ] \ldots ]] \in \langle \g^+_{[M]} \rangle$.  If $n_1 + \cdots + n_k \le 0$, then $[x_1, [x_2, \ldots [ x_{k-1}, x_k ] \ldots ]] \in \langle \g^-_{[M]} \rangle$.
\end{lem}
\begin{proof}
The proof is by induction on $k$.   If $n_1 + \cdots + n_k = 0$, then we have $[x_1, [x_2, \ldots [ x_{k-1}, x_k ] \ldots ]] \in \g_0 \subseteq \g^+_{[M]} \cap \g^-_{[M]}$, so the result is obvious.  Thus we assume $n_1 + \cdots + n_k \neq 0$.  

First suppose $n_1 + \cdots + n_k > 0$.  If $n_i \ge 0$ for all $i$, then the result is immediate, so we assume some $n_i$ is negative.  Since the sum $n_1 + \cdots + n_k$ is positive, $\{ n_1, \ldots, n_k \}$ contains both a positive and a negative element.

Suppose for now that $n_k \le 0$.   As $\{ n_1, \ldots, n_k \}$ contains both a positive and a negative element, we may let $\ell$ be maximal such that $n_\ell > 0$ (noting that $\ell < k$).  Observe that  
\begin{align*}
[x_\ell, [ x_{\ell + 1}, \ldots [ x_{k-1}, x_k ] \ldots ]] &=  \ad_{x_\ell} ( [ x_{\ell + 1}, \ldots [ x_{k-1}, x_k ] \ldots ] ) \\
&= \sum_{\ell < i \le k} [ x_{\ell + 1}, \ldots , [ \ad_{x_\ell} (x_i) , \ldots , [x_{k-1}, x_k ] \ldots ] \ldots ].
\end{align*}
Now $\ad_{x_\ell} (x_i) \in \g_{n_\ell + n_i}$, and by our choice of $\ell$, we know that $-M \le n_{\ell} + n_i \le M$.  Define $x_i' := \ad_{x_\ell}(x_i) \in \g_{n_i'}$, where $n_i' = n_\ell + n_i$ and $-M \le n_i' \le M$.  We now have 
\begin{align*}
[x_\ell, [ x_{\ell + 1}, \ldots [ x_{k-1}, x_k ] \ldots ]] = \sum_{\ell < i \le k} [ x_{\ell + 1}, \ldots , [ x_i' , \ldots , [x_{k-1}, x_k ] \ldots ] \ldots ],
\end{align*}
where the commutator on the left is constructed from $k - \ell + 1$ elements, and the commutators on the right are constructed from $k - \ell$ elements.   Thus, the original commutator $[x_1, [x_2, \ldots [ x_{k-1}, x_k ] \ldots ]]$
can be expressed as a sum of commutators $[y_1, [y_2, \ldots [ y_{k-2}, y_{k-1} ] \ldots ]]$ where 
\begin{enumerate}
\item $y_i \in \g_{m_i}$ for some $m_i \in \z$ with $|m_i| \le M$, and 

\item $m_1 + \cdots + m_{k-1} = n_1 + \cdots + n_k$.
\end{enumerate}
Then $[y_1, [y_2, \ldots [ y_{k-2}, y_{k-1} ] \ldots ]] \in \langle \g^+_{[M]} \rangle$ by induction; and it follows that 
$$[x_1, [x_2, \ldots [ x_{k-1}, x_k ] \ldots ]] \in \langle \g^+_{[M]} \rangle.$$
We can apply similar arguments to the case $n_1 + \cdots + n_k > 0$ where $n_k > 0$; and to the case $n_1 + \cdots + n_k <0$.
\end{proof}

\begin{lem}\label{lem:finGenPosPart}
Let $\g$ be the insertion-elimination Lie algebra.  If $\g$ is finitely generated as a Lie algebra, then $\bigoplus_{n \ge 0} \g_n$ is also finitely generated Lie algebra.
\end{lem}
\begin{proof}
Suppose that there is a finite generating set $\mathcal G$ for $L$.  Then there is some positive integer $M$ such that $\g_{[M]} = \bigoplus_{i= -M}^M \g_i$ contains $\mathcal G$ and therefore generates $\g$.

Now let $x \in \g_n$, $n > 0$.  By assumption, $x \in \langle \g_{[M]} \rangle$.  Then by Lemma \ref{lem:reorderBrackets}, $x$ can be written as a sum of elements of the form $[x_1, [x_2, \ldots [ x_{k-1}, x_k ] \ldots ]]$ for some $x_1, \ldots, x_k \in \g_{[M]}$.  We may further assume the elements $x_i$ are homogeneous with respect to the $\z$-grading on $\g$.   Since the graded subspaces $\g_i \subseteq L$ are linearly independent, and $x \in \g_n$, it is no loss to assume that the summands $[x_1, [x_2, \ldots [ x_{k-1}, x_k ] \ldots ]]$ belong to $\g_n$.   Now by Lemma \ref{lem:prodReducePos}, each $[x_1, [x_2, \ldots [ x_{k-1}, x_k ] \ldots ]]$ belongs to $\langle \g^+_{[M]} \rangle$, so in fact $x \in \langle \g^+_{[M]} \rangle$.  
\end{proof}

\begin{lem}\label{lem:posNotFinGen} 
Let $\g = \g^- \oplus \h \oplus \g^+$ denote the insertion-elimination Lie algebra.  The subalgebra $\g^+$ (or equivalently, the subaglebra $\h \oplus \g^+$) is not finitely generated as a Lie algebra.
\end{lem}

In the proof below, $\ell_n \in \T_n$ is the ladder with $n$ vertices.  (For example, $\ell_3 = \psset{levelsep=0.3cm, treesep=0.3cm} 
\pstree{\Tr{$\bullet$}}{\pstree{\Tr{$\bullet$}}{\Tr{$\bullet$}}}$.)  Note that for $t \in \T$, $\depth (t) = |t|-1$ if and only if $t = \ell_{|t|}$.  

\begin{proof}
It is enough to argue that, for any given $m >0$, $D_{\ell_{m+1}}^+ \not\in \langle \g^+_{[m]} \rangle$.  In particular, we argue that if $s, t \in \T$ with $[D_s^+, D_t^+ ] \in \g_{m+1}$, then the coefficient of $D_{\ell_{m+1}}^+$ in $[D_s^+, D_t^+ ]$ is zero. 

Recall 
$$[D_s^+, D_t^+ ]=\sum_{v \in V(t)} D^+_{t \cup_v s} - \sum_{w \in V(s)} D^+_{s \cup_v t }.$$

If either $s$ or $t$ is not a ladder, then for any $v \in V(t)$ or $v' \in V(s)$, the trees $t \cup_v s$ and $s \cup_{v'} t$ are not ladders.  On the other hand, if $s$ and $t$ are both ladders, it's straightforward to check from the commutator relation that the coefficient of $D_{\ell_{m+1}}^+$ in $[D_s^+, D_t^+]$ is zero.
\end{proof}

From Lemmas \ref{lem:finGenPosPart} and \ref{lem:posNotFinGen}, we have the following result.

\begin{prop}\label{prop:IE-NotFinGen}
The insertion-elimination algebra $\g$ is not finitely generated as a Lie algebra. 
\end{prop}

\begin{exam}\label{exam:subalgInsElim}
Proposition \ref{prop:IE-NotFinGen} can be used to show the existence of many infinite-dimensional, finitely generated subalgebras of $\g$ that fit into the framework described in this paper.

In general, consider any subalgebra $\mathfrak a \subseteq \g$ that contains the element $d$ and is invariant under the automorphism $\tau_0$ (i.e. $\tau_0 (\mathfrak a) \subseteq \mathfrak a$) defined in (\ref{eqn:tau0}).  The assumption that $d \in \mathfrak a$ implies that $\mathfrak a = \bigoplus_{n \in \z} \mathfrak a_n$, where $\mathfrak a_n = \mathfrak a \cap \g_n$; and clearly $\dim \mathfrak a_n < \infty$.  As both $\g^+$ and $\g^-$ are completely self-centralizing Lie algebras, both $\mathfrak a^+ = \bigoplus_{n > 0} \mathfrak a_n$ and $\mathfrak a^- = \bigoplus_{n < 0} \mathfrak a_n$ are also completely self-centralizing.  Since $\mathfrak a$ is invariant under $\tau_0$, it is evident that $\mathfrak a_n \neq 0$ if and only if $\mathfrak a_{-n} \neq 0$. Thus $\mathfrak a$ satifsy the assumptions of Section \ref{subsec:genGradedLie}, and it follows that Proposition \ref{prop:finSubsGeneral}, Proposition \ref{prop:autPreserveCartan}, and Proposition \ref{prop:DerDirectSum} can be applied any such subalgebra.

More specifically, suppose $S \subseteq \{ D_t^+ \mid t \in \T \} \subseteq \g^+$, and let $\g_S$ denote the Lie subalgebra of $\g$ generated by $S$ and $d$ and $\tau_0 (S)$.  Then $\g_S$ is clearly invariant under $\tau_0$.  Moreover, if $|S| < \infty$, Proposition \ref{prop:IE-NotFinGen} implies that $\g_S$ is a proper subalgebra of $\g$. 
\qed
\end{exam}

\end{document}